\documentclass[11pt]{article}
\usepackage{amsmath,amssymb,amsthm,graphicx}
\usepackage[mathscr]{eucal}
\usepackage{fullpage}
\usepackage{color} 
\usepackage[english]{babel}
\usepackage[latin1]{inputenc} 

\newtheorem{theorem}{Theorem}[section]
\newtheorem{proposition}[theorem]{Proposition}
\newtheorem{corollary}[theorem]{Corollary}
\newtheorem{lemma}[theorem]{Lemma}
\newtheorem{example}{Example}[section]

\def\ker{\mathop{\mathrm{Ker}}\nolimits}
\def\im{\mathop{\mathrm{Im}}\nolimits}
\def\rank{\mathop{\mathrm{rank}}\nolimits}
\def\reg{\mathop{\mathrm{Reg}}\nolimits}
\def\fix{\mathop{\mathrm{Fix}}\nolimits}
\def\N{\mathbb N}
\def\Z{\mathbb Z}
\def\R{\mathbb R}
\def\Q{\mathbb Q}
\def\U{\mathbb U}

\def\O{\mathcal{O}}
\def\T{\mathcal{T}}
\def\F{\mathcal{F}}
\def\FO{\mathcal{FO}}
\def\X{\mathscr{X}}
\def\Zn{\overline{\mathbb Z}}
\def\Nn{\overline{\mathbb N}}

\newcommand{\lastpage}{\addresss}

\newcommand{\addresss}{\small \sf  
\noindent{\sc V\'\i tor H. Fernandes}, 
Departamento de Matem\'atica, 
Faculdade de Ci\^encias e Tecnologia, 
Universidade Nova de Lisboa, 
Monte da Caparica, 
2829-516 Caparica, 
Portugal; 
also: 
Centro de \'Algebra da Universidade de Lisboa, 
Av. Prof. Gama Pinto 2, 
1649-003 Lisboa, 
Portugal; 
e-mail: vhf@fct.unl.pt

\medskip

\noindent{\sc Preeyanuch Honyam},  
Department of Mathematics, Faculty of Science, Chiang Mai University, Chiang Mai 50200, Thailand; 
e-mail: preeyanuch\_h@hotmail.com

\medskip

\noindent{\sc Teresa M. Quinteiro}, 
Instituto Superior de Engenharia de Lisboa, 
Rua Conselheiro Em\'\i dio Navarro 1, 
1950-062 Lisboa, 
Portugal; 
also: 
Centro de \'Algebra da Universidade de Lisboa, 
Av. Prof. Gama Pinto 2, 
1649-003 Lisboa, 
Portugal;
e-mail: tmelo@dec.isel.ipl.pt

\medskip

\noindent{\sc Boorapa Singha},  
School of Mathematics and Statistics, 
Faculty of Science and Technology, 
Chiang Mai Rajabhat University, 
Chiang Mai 50300, 
Thailand;
email: boorapa\_sin@cmru.ac.th 
}

\title{On semigroups of endomorphisms of a chain with restricted range}

\author{V\'\i tor H. Fernandes\footnote{This work was developed within the research activities of Centro de Álgebra da Universidade de Lisboa, FCT´s project PEst-OE/MAT/UI0143/2013, and of Departamento de Matemática da Faculdade de Ciências e Tecnologia da Universidade Nova de Lisboa.},~ 
Preeyanuch Honyam, 
Teresa M. Quinteiro\footnote{This work was developed within the research activities of Centro de Álgebra da Universidade de Lisboa, FCT´s project PEst-OE/MAT/UI0143/2013, and of Instituto Superior de Engenharia de Lisboa.}~ 
and Boorapa Singha}

\begin{document}
\maketitle

\begin{abstract}
Let $X$ be a finite or infinite chain and let $\O(X)$ be the monoid of all  endomorphisms of $X$. 
In this paper, we describe the largest regular subsemigroup of $\O(X)$ and Green's relations on $\O(X)$. 
In fact, more generally, if $Y$ is a nonempty subset of $X$ and $\O(X,Y)$ the subsemigroup of $\O(X)$ of all elements with range contained in $Y$, 
we characterize the largest regular subsemigroup of $\O(X,Y)$ and Green's relations on $\O(X,Y)$. 
Moreover, for finite chains, we determine when two semigroups of the type $\O(X,Y)$ are isomorphic and calculate their ranks. 
\end{abstract}

\medskip



\section*{Introduction and preliminaries}

Let $X$ be a nonempty set and denote by $\T(X)$ the monoid of all
(full) transformations of $X$ (under composition).  

Throughout this paper, we will represent a chain only by its support set and, as usual, 
its order by the symbol $\le$. 

Now, let $X$ be a chain.  
A function $\theta: A\longrightarrow X$ from a subchain $A$ of $X$ into $X$ is said to be \textit{order-preserving} if $x\le y$ implies $x\theta\le y\theta$, for all $x,y\in A$. 
Notice that, given two subchains $A$ and $B$ of $X$ and an \textit{order-isomorphism} 
(i.e. an order-preserving bijection) $\theta: A\longrightarrow B$,  
then the inverse function $\theta^{-1}: B\longrightarrow A$ is also an order-isomorphism. In this case, the subchains $A$ and $B$ are called \textit{order-isomorphic}. 
We denote by $\O(X)$ the submonoid of $\T(X)$ of all (order) \textit{endomorphisms} of $X$, 
i.e. of all order-preserving transformations of $X$. 

\smallskip 

For a finite chain $X$, it is well known, and clear, that $\O(X)$ is a regular semigroup. The problem for an infinite chain $X$ is much more involved. 
Nevertheless, more generally, a characterization of those posets $P$ for which the semigroup of all endomorphisms of $P$ is regular was done by A\v\i zen\v stat in 1968 
\cite{Aizenstat:1968} and, independently, by Adams and Gould in 1989 \cite{Adams&Gould:1989}. 
Returning to the finite case, 
if $X$ is a chain with $n$ elements, e.g. $X=\{1<2<\cdots<n\}$,  
we usually denote the monoid $\O(X)$ by $\O_n$.  
This monoid has been extensively studied since the sixties.
In fact, in 1962, A\v\i zen\v stat \cite{Aizenstat:1962,Aizenstat:1962b} 
showed that the
congruences of $\O_n$ are exactly Rees congruences and gave a monoid presentation for $\O_n$, 
in terms of $2n-2$ idempotent generators, from which it can be deduced that the only non-trivial automorphism of $\O_n$ where $n>1$ is 
the one given by conjugation by permutation $(1\; n)(2\;n-1)\cdots(\lfloor n/2\rfloor\; \lceil n/2\rceil+1)$. 
In 1971, Howie \cite{Howie:1971} calculated the cardinal and the number of idempotents of $\O_n$
and later (1992), jointly with Gomes \cite{Gomes&Howie:1992}, 
determined its rank and idempotent rank. 
More recently, Fernandes et al.~\cite{Fernandes&al:2010} described the endomorphisms of the semigroup $\O_n$ by showing that there are three types of endomorphism: automorphisms, constants, and a certain type of endomorphism with two idempotents in the image. 
The monoid $\O_n$ also played a main role in several other papers 
\cite{
Almeida&Volkov:1998,Fernandes:1997,Fernandes:2002,Fernandes&Volkov:2010,
Higgins:1995,Repnitskii&Vernitskii:2000,Repnitskii&Volkov:1998,Vernitskii&Volkov:1995}  
where the central topic concerns 
the problem of the decidability of the pseudovariety generated by the family 
$\{\O_n\mid n\in\N\}$. 
This question was posed 
by J.-E. Pin in 1987 in the ``Szeged International Semigroup Colloquium''
and, as far as we know, is still unanswered. 

\smallskip 

Given a nonempty subset $Y$ of $X$, we denote by $\T(X,Y)$ the subsemigroup 
$\{\alpha\in \T(X)\mid \im(\alpha)\subseteq Y\}$ of $\T(X)$ of all elements with range (image) restricted to $Y$. 

In 1975, Symons \cite{Symons:1975} introduced and studied the semigroup 
$\mathcal{T}(X,Y)$. 
He described all the automorphisms of $\mathcal{T}(X,Y)$ and also determined when 
two semigroups of this type are isomorphic. 
In \cite{Nenthein&et.al:2005}, Nenthein et al. characterized the regular elements of $\mathcal{T}(X,Y)$ and, in \cite{Sanwong&Sommanee:2008},
Sanwong and Sommanee obtained the largest regular subsemigroup of $\mathcal{T}(X,Y)$ and showed that this subsemigroup determines Green's relations 
on $\mathcal{T}(X,Y)$. Moreover, they also determined a class of maximal inverse subsemigroups of this semigroup. 
Later, in 2009, all maximal and minimal congruences on $\mathcal{T}(X,Y)$ were described by Sanwong et al.~\cite{Sanwong&et.al:2009}.
Recently, all the ideals of $\mathcal{T}(X,Y)$ were obtained by Mendes-Gon\c calves and Sullivan in \cite{Goncalves&Sullivan:2011} and, for a finite set $X$, Fernandes and Sanwong computed the rank of $\mathcal{T}(X,Y)$  \cite{Fernandes&Sanwong:2012}. 
On the other hand, in \cite{Sullivan:2008}, Sullivan considered the linear counterpart of $\mathcal{T}(X,Y)$, that is the semigroup $\mathcal{T}(V,W)$
which consists of all linear transformations from a vector space $V$ into a fixed subspace $W$ of $V$, and described its Green's relations and ideals. 

\smallskip

In this paper, for a chain $X$ and a nonempty subset $Y$ of $X$, 
we consider the order-preserving counterpart of the semigroup $\mathcal{T}(X,Y)$, namely the semigroup 
$\O(X,Y)=\mathcal{T}(X,Y)\cap\O(X) = \{\alpha\in \O(X)\mid \im(\alpha)\subseteq Y\}$. 
If $X$ is a (finite) chain, say $X=\{1<2<\cdots<n\}$, 
we denote $\O(X,Y)$ simply by $\O_n(Y)$. 

A description of the regular elements of $\O(X,Y)$ and a characterization of the regular semigroups of this type were given by Mora and Kemprasit in  \cite{Mora&Kemprasit:2010}. Here, in Section \ref{green}, we describe the largest regular subsemigroup of $\O(X,Y)$ (particularly of $\O(X)$) and Green's relations on $\O(X,Y)$. In special, we obtain descriptions for Green's relations on $\O(X)$, which surprisingly, as far as we know, were not characterized before. In Section \ref{iso}, for finite chains, we determine when two semigroups of the type $\O(X,Y)$ are isomorphic. Finally, in Section \ref{rank}, we calculate the rank of the semigroups $\O_n(Y)$, for each nonempty subset $Y$ of $\{1,2,\ldots,n\}$. 

\medskip 

For general background on Semigroup Theory and standard notation, 
we refer the reader to Howie's book \cite{Howie:1995}.

\section{Regularity and Green's relations}\label{green}

Let $X$ be any chain and let $Y$ be a nonempty subset of $X$. 

The following useful regularity criterion for the elements of $\O(X)$ was proved  in \cite[Theorem 2.4]{Mora&Kemprasit:2010} by Mora and Kemprasit. 

\begin{theorem}\label{regsOX}
Let $X$ be a chain and let $\alpha\in\O(X)$. Then $\alpha\in\reg(\O(X))$ if and only if the following three conditions hold:
\begin{enumerate}
\item If $\im(\alpha)$ has an upper bound in $X$, then $\max(\im(\alpha))$ exists; 
\item If $\im(\alpha)$ has a lower bound in $X$, then $\min(\im(\alpha))$ exists;
\item If $x\in X\setminus\im(\alpha)$ is neither an upper bound nor a lower bound of $\im(\alpha)$, 
then either $\max\{a\in\im(\alpha)\mid a<x\}$ or $\min\{a\in\im(\alpha)\mid x<a\}$ exists. 
\end{enumerate}
\end{theorem}

Based on this theorem,  Mora and Kemprasit \cite{Mora&Kemprasit:2010} deduced several previous known results. 
For instance, that $\O(\Z)$ is regular while $\O(\Q))$ and $\O(\R)$ are not regular, by considering their usual orders. 
See also \cite{Adams&Gould:1989,Aizenstat:1968,Kemprasit&Changphas:2000,Kim&Kozhukhov:2008}. 

Here arises a natural question: describe the maximal regular subsemigroups of $\O(X)$. Our first result applies Theorem \ref{regsOX} to answer this question. 

\begin{theorem}\label{regOX}
Let $X$ be any chain. Then $\reg(\O(X))$ is a subsemigroup of $\O(X)$. Consequently, $\reg(\O(X))$ is the largest regular subsemigroup of $\O(X)$. 
\end{theorem}
\begin{proof}
Let $\alpha$ and $\beta$ be two regular elements of $\O(X)$. 

Assume that $\im(\alpha\beta)$ has an upper bound $x\in X$. 
If both $\im(\alpha)$ and $\im(\beta)$ have no upper bounds in $X$, 
then there exist $a,b\in X$ such that $b\beta>x$ and $a\alpha>b$, 
whence $a\alpha\beta\ge b\beta>x$, which is a contradiction. 
Hence, $\im(\alpha)$ or $\im(\beta)$ has an upper bound in $X$. 
If $\im(\alpha)$ has an upper bound in $X$ then, by Theorem \ref{regsOX}, there exists $m=\max(\im(\alpha))$ and so, 
clearly, we have $m\beta=\max((\im(\alpha)\beta)=\max(\im(\alpha\beta))$. 
On the other hand, suppose that $\im(\alpha)$ has no upper bounds in $X$. 
Then $\im(\beta)$ must have an upper bound in $X$ and so, 
by Theorem \ref{regsOX}, there exists $m=\max(\im(\beta))$. 
Let $a, b\in X$ be such that $b\beta=m$ and $a\alpha > b$. 
Then $a\alpha\beta\ge b\beta=m$, whence $a\alpha\beta= b\beta=m$ and so we also have 
$m=\max(\im(\alpha\beta))$. Thus, in all cases, we proved that $\max(\im(\alpha\beta))$ exists. 

Dually, assuming that $\im(\alpha\beta)$ has a lower bound in $X$, we 
may show that $\min(\im(\alpha\beta))$ exists. 

Now, suppose that $x\in X\setminus\im(\alpha\beta)$ is neither an upper bound nor a lower bound of $\im(\alpha\beta)$. Let $a,b\in X$ be such that $a\alpha\beta<x<b\alpha\beta$. 
Let us consider two cases: $x\in\im(\beta)$ and $x\in X\setminus\im(\beta)$. 

First, suppose that $x\in\im(\beta)$. Then $x=y\beta$, for some $y\in X$. 
Since $a\alpha\beta<x<b\alpha\beta$, we may deduce that $a\alpha<y<b\alpha$, 
whence $y$ is neither an upper bound nor a lower bound of $\im(\alpha)$.  
Moreover, as $x\in X\setminus\im(\alpha\beta)$, then $y\in X\setminus\im(\alpha)$. 
Hence, by Theorem \ref{regsOX}, either $\max\{t\in\im(\alpha)\mid t<y\}$ or $\min\{t\in\im(\alpha)\mid y<t\}$ exists. 
Clearly, it follows that 
$(\max\{t\in\im(\alpha)\mid t<y\})\beta=\max(\{t\in\im(\alpha)\mid t<y\}\beta)=\max\{z\in\im(\alpha\beta)\mid z<x\}$ 
or $(\min\{t\in\im(\alpha)\mid y<t\})\beta=\min(\{t\in\im(\alpha)\mid y<t\}\beta)=\min\{z\in\im(\alpha\beta)\mid x<z\}$ exists.

Secondly, suppose that $x\in X\setminus\im(\beta)$. 
Since $(a\alpha)\beta<x<(b\alpha)\beta$, then $x$ is neither an upper bound nor a lower bound of $\im(\beta)$ and so, 
by Theorem \ref{regsOX}, either $\max\{z\in\im(\beta)\mid z<x\}$ or $\min\{z\in\im(\beta)\mid x<z\}$ exists.   

Assume that $m=\max\{z\in\im(\beta)\mid z<x\}$ exists. Notice that $m\in\im(\beta)$ and $m<x$. Let $y\in X$ be such that $m=y\beta$.  

If $m\in\im(\alpha\beta)$, since $\im(\alpha\beta)\subseteq\im(\beta)$, we also have  
$m=\max\{z\in\im(\alpha\beta)\mid z<x\}$. 

On the other hand, suppose that $m\not\in\im(\alpha\beta)$. Then $y\not\in\im(\alpha)$. 
If $y\le a\alpha$ then $m=y\beta\le a\alpha\beta < x$ and so, 
since $a\alpha\beta\in\im(\beta)$, we have $m=y\beta=a\alpha\beta\in\im(\alpha\beta)$, 
which is a contradiction. Thus $a\alpha < y$. If $b\alpha \le y$ then $b\alpha\beta\le y\beta=m < x$, which is also a contradiction. Therefore $a\alpha < y < b\alpha$ and so 
$y\in X\setminus\im(\alpha)$ is neither an upper bound nor a lower bound of $\im(\alpha)$.
Hence, by Theorem \ref{regsOX}, either $\max\{t\in\im(\alpha)\mid t<y\}$ or $\min\{t\in\im(\alpha)\mid y<t\}$ exists. 
As above, it follows that 
$(\max\{t\in\im(\alpha)\mid t<y\})\beta=\max(\{t\in\im(\alpha)\mid t<y\}\beta)=\max\{z\in\im(\alpha\beta)\mid z<m\}$ 
or $(\min\{t\in\im(\alpha)\mid y<t\})\beta=\min(\{t\in\im(\alpha)\mid y<t\}\beta)=\min\{z\in\im(\alpha\beta)\mid m<z\}$ exists. 

Now, first of all suppose that $m'=\max\{z\in\im(\alpha\beta)\mid z<m\}$ exists. 
Let $z\in\im(\alpha\beta)$ be such that $z<x$. Then $z\le m$ and so $z<m$, since $m\not\in\im(\alpha\beta)$. It follows that $z\le m'$. As $m'\in\im(\alpha\beta)$, 
we proved that we also have $m'=\max\{z\in\im(\alpha\beta)\mid z<x\}$. 

Secondly, suppose that $m'=\min\{z\in\im(\alpha\beta)\mid m<z\}$. 
Then $m'\in\im(\alpha\beta)$ and $m < m'$.  
If $m'< x$ then $m'\le m$, which is a contradiction. Thus $x\le m'$ and, in fact, 
$x < m'$, since $x\not\in\im(\alpha\beta)$. 
So $m'\in \{z\in\im(\alpha\beta)\mid x<z\}$. 
Next, let $z\in\im(\alpha\beta)$ be such that $x<z$.
Since $m<x$, it follows that $m<z$, whence $m'\le z$. 
So $m'=\min\{z\in\im(\alpha\beta)\mid x<z\}$. 

Therefore, assuming that $\max\{z\in\im(\beta)\mid z<x\}$ exists, we proved that either 
$\max\{z\in\im(\alpha\beta)\mid z<x\}$ or $\min\{z\in\im(\alpha\beta)\mid x<z\}$ exists. 
Dually, assuming that $\min\{z\in\im(\beta)\mid x<z\}$ exists, we may reach the same conclusion, as required. 
\end{proof}

Now, let $\F(X,Y)=\{\alpha\in\T(X,Y)\mid \im(\alpha)\subseteq Y\alpha\}=\{\alpha\in\T(X,Y)\mid \im(\alpha)=Y\alpha\}$. Clearly, 
the set $\F(X,Y)$ is a right ideal of $\T(X,Y)$. Moreover, it is easy to show that $\F(X,Y)=\reg(\T(X,Y))$ and so it is the largest regular subsemigroup of $\T(X,Y)$.  
See \cite{Sanwong&Sommanee:2008} (and also \cite{Nenthein&et.al:2005}). 

Also in \cite[Theorem 3.1]{Mora&Kemprasit:2010} Mora and Kemprasit showed:

\begin{theorem}\label{regOXY}
Let $X$ be any chain and let $Y$ be a nonempty subset of $X$. Then 
$$
\reg(\O(X,Y))=\reg(\T(X,Y))\cap\reg(\O(X))=\{\alpha\in\reg(\O(X))\mid \im(\alpha)=Y\alpha\subseteq Y\}.
$$
\end{theorem}

Since $\F(X,Y)$ is a subsemigroup of $\T(X,Y)$, in view of Theorems \ref{regOX} and \ref{regOXY}, we have the following immediate corollary. 

\begin{corollary}
Let $X$ be any chain and let $Y$ be a nonempty subset of $X$. Then 
$\reg(\O(X,Y))$ is the largest regular subsemigroup of $\O(X,Y)$. 
\end{corollary}

Consider the subset 
$$
\FO(X,Y)=\{\alpha\in\O(X,Y)\mid \im(\alpha)\subseteq Y\alpha\}=\{\alpha\in\O(X,Y)\mid \im(\alpha)=Y\alpha\}
$$ 
of $\O(X,Y)$. 
It is clear that $\FO(X,Y)$ is a right ideal of $\O(X,Y)$ containing $\reg(\O(X,Y))$. 
We may wonder whether $\reg(\O(X,Y))=\FO(X,Y)$. Naturally, if $\O(X)$ is regular 
then trivially the equality holds (see also \cite[Theorem 3.6]{Mora&Kemprasit:2010}). However, this is not the case in general. For instance, consider $X=\R$ equipped with the usual order, $Y=\left]-\infty,0\right]$ and $\alpha\in\T(X)$ defined by 
$$
x\alpha=\left\{
\begin{array}{ll}
e^x-1 & \mbox{if $x\le0$}\\
0 & \mbox{if $x>0$.}
\end{array}
\right. 
$$
Then $\alpha\in\O(X)$ and $\im(\alpha)=\left]-1,0\right]=Y\alpha\subseteq Y$, whence $\alpha\in\FO(X,Y)$. On the other hand, $\im(\alpha)$ has lower bounds in $X$ but no minimum. Thus $\alpha\not\in\reg(\O(X,Y))$. 

\bigskip

Now, recall the well known descriptions of Green's relations $\mathscr{L}$, 
$\mathscr{R}$ and $\mathscr{D}$ on $\T(X)$ (see e.g. \cite[Page 63]{Howie:1995}): 
\begin{enumerate}
\item $\alpha\mathscr{L}\beta$ if and only if $\im(\alpha)=\im(\beta)$,  
\item $\alpha\mathscr{R}\beta$ if and only if $\ker(\alpha)=\ker(\beta)$ and 
\item $\alpha\mathscr{D}\beta$ if and only if $|\im(\alpha)|=|\im(\beta)|$, 
\end{enumerate}
for all $\alpha,\beta\in\T(X)$. Moreover, in $\T(X)$, we have $\mathscr{J}=\mathscr{D}$.  

\medskip 

Next, we present descriptions for Green's relations on $\O(X,Y)$ 
and, in particular, on $\O(X)$. 
We start with $\mathcal{L}$. First, observe that 
Sanwong and Sommanee \cite[Theorem 3.2]{Sanwong&Sommanee:2008} showed that, 
for $\alpha,\beta\in\T(X,Y)$, we have $\alpha\mathscr{L}\beta$ in $\T(X,Y)$ if and only if 
either $\alpha=\beta$ or $\alpha,\beta\in\F(X,Y)$ and $\im(\alpha)=\im(\beta)$. 
An analogous result holds for $\O(X,Y)$. 

\begin{proposition}\label{LOXY}
Let $X$ be a chain and let $Y$ be a nonempty subset of $X$.  
Then, for all $\alpha,\beta\in\O(X,Y)$, 
we have $\alpha\mathscr{L}\beta$ in $\O(X,Y)$ if and only if 
either $\alpha=\beta$ or $\alpha,\beta\in\FO(X,Y)$ and $\im(\alpha)=\im(\beta)$. 
\end{proposition}
\begin{proof} 
Suppose $\alpha\mathscr{L}\beta$ in $\O(X,Y)$. 
Since $\alpha\mathscr{L}\beta$ in $\O(X,Y)$ implies $\alpha\mathscr{L}\beta$ in $\T(X,Y)$, 
if $\alpha\ne\beta$ then $\alpha,\beta\in\F(X,Y)$ and $\im(\alpha)=\im(\beta)$. 
Hence, either $\alpha=\beta$ or $\alpha,\beta\in\FO(X,Y)$ and $\im(\alpha)=\im(\beta)$.  

Conversely, assume that $\alpha,\beta\in\FO(X,Y)$ and $\im(\alpha)\subseteq\im(\beta)$. 
For each $a\in\im(\alpha)$ choose an element $u_a\in a\beta^{-1}\cap Y$. 
Define a transformation $\gamma$ of $X$ by $x\gamma=u_{x\alpha}$, for all $x\in X$. 
Hence, for all $x\in X$, we have $x\gamma\beta=u_{x\alpha}\beta=x\alpha$, 
i.e. $\alpha=\gamma\beta$. Clearly, we also have $\im(\gamma)\subseteq Y$. 
Furthermore, $\gamma\in\O(X,Y)$. 
In fact, let $x,y\in X$ be such that $x\le y$. 
Then $x\alpha\le y\alpha$. If $u_{x\alpha}\ge u_{y\alpha}$ then 
$x\alpha=u_{x\alpha}\beta\ge u_{y\alpha}\beta=y\alpha$, 
whence $x\alpha=y\alpha$ and so $u_{x\alpha}=u_{y\alpha}$.
Thus $x\gamma=u_{x\alpha}\le u_{y\alpha}=y\gamma$. 

Similarly, by assuming that $\alpha,\beta\in\FO(X,Y)$ and  $\im(\beta)\subseteq\im(\alpha)$, we may show the existence of a transformation $\lambda\in\O(X,Y)$ such that $\beta=\lambda\alpha$. Therefore, $\alpha,\beta\in\FO(X,Y)$ and $\im(\alpha)=\im(\beta)$ implies $\alpha\mathscr{L}\beta$ in $\O(X,Y)$, as required. 
\end{proof}

In particular for $Y=X$, a simpler statement can be presented:  

\begin{corollary}\label{LOX}
Let $X$ be a chain and let $\alpha,\beta\in\O(X)$. 
Then $\alpha\mathscr{L}\beta$ in $\O(X)$ if and only if $\im(\alpha)=\im(\beta)$.
\end{corollary}

Notice that relation $\mathcal{L}$ in $\O(X)$ is just the restriction of relation $\mathcal{L}$ in $\T(X)$, despite $\O(X)$ may be non-regular. 

Before presenting a description for relation $\mathcal{R}$, we introduce the notion of completable order-preserving function and provide an alternative characterization, which helps to understand its nature. 

We say that an order-preserving function $\theta: A\longrightarrow B$ from a subchain $A$ of $X$ into a subchain $B$ of $Y$ is \textit{completable} in $\O(X,Y)$ if there exists $\gamma\in\O(X,Y)$ such $a\gamma=a\theta$, 
for all $a\in A$. To such $\gamma\in\O(X,Y)$ (not necessarily unique) 
we designate \textit{a complete extension} of $\theta$ in $\O(X,Y)$.  
An order-isomorphism $\theta: A\longrightarrow B$, with $A$ and $B$ two subchains of $Y$, is said to be \textit{bicompletable} in $\O(X,Y)$ if both $\theta$ and its inverse 
$\theta^{-1}: B\longrightarrow A$ are completable in $\O(X,Y)$. 

Observe that, an order-isomorphism between two subchains may be completable but not bicompletable. 
For example, with $Y=X=\R$ equipped with the usual order, in $\O(\R)$ the order-isomorphism 
$\R\longrightarrow \left]0,+\infty\right[$, $x\mapsto e^x$, is trivially completable while its inverse $\left]0,+\infty\right[\longrightarrow\R $, $x\mapsto \log(x)$, is clearly non-completable. 

Recall that a subset $I$ of $X$ (including the empty set) is called an order ideal of $X$ if $x\le a$ implies $x\in I$, for all $x\in X$ and all $a\in I$. 
The following characterization of completable order-preserving functions may be useful in practice. 

\begin{proposition}\label{compl}
Let $X$ be a chain and let $Y$ be a nonempty subset of $X$.  
Let $A$ be a subchain of $X$. 
An order-preserving function $\theta: A\longrightarrow Y$ is completable in $\O(X,Y)$ if and only if $\{x\in X\mid a<x<b, ~\mbox{for all $a\in I$ and $b\in A\setminus I$}\}\neq\emptyset$ implies 
$\{y\in Y\mid a\theta\le y\le b\theta, ~\mbox{for all $a\in I$ and $b\in A\setminus I$}\}\neq\emptyset$, for all order ideal $I$ of $A$. 
\end{proposition}

Observe that, 
if $I=\emptyset$ (respectively, $I=A$), the set 
$$\{x\in X\mid a<x<b, ~\mbox{for all $a\in I$ and $b\in A\setminus I$}\}$$ 
should naturally be understood as $\{x\in X\mid x<b, ~\mbox{for all $b\in A$}\}$ 
(respectively, $\{x\in X\mid a<x, ~\mbox{for all $a\in A$}\}$). 
For the set $\{y\in Y\mid a\theta\le y\le b\theta, ~\mbox{for all $a\in I$ and $b\in A\setminus I$}\}$ we make similar assumptions. 

\begin{proof}[Proof of Proposition \ref{compl}]
The direct implication is clear. In order to prove the converse implication, 
for each order ideal $I$ of $A$ such that the set $\{x\in X\mid a<x<b, ~\mbox{for all $a\in I$ and $b\in A\setminus I$}\}$ is nonempty, choose an element  
$v_I$ belonging to the set $$\{y\in Y\mid a\theta\le y\le b\theta, 
~\mbox{for all $a\in I$ and $b\in A\setminus I$}\}.$$ 

Let $x\in X\setminus A$ and let $I_x=\{a\in A\mid a < x\}$. Clearly, $I_x$ is an order ideal of $A$ and $a < x < b$, for all $a\in I_x$ and $b\in A\setminus I_x$. 
Thus, we have an element $v_{I_x}\in Y$ verifying $a\theta\le v_{I_x}\le b\theta$, 
for all $a\in I_x$ and $b\in A\setminus I_x$. 

Define a transformation $\gamma$ of $X$ by 
$$
x\gamma=\left\{
\begin{array}{ll}
x\theta & \mbox{if $x\in A$}\\
v_{I_x} & \mbox{if $x\in X\setminus A$.}
\end{array}
\right. 
$$
Clearly, if $\gamma\in\O(X,Y)$ then $\gamma$ is a complete extension of $\theta$. 
Thus, let us take $x,y\in X$ such that $x\le y$. Next, we consider four cases. 
If $x,y\in A$, then $x\gamma=x\theta\le y\theta=y\gamma$. 
If $x\in X\setminus A$ and $y\in A$, then $y\in A\setminus I_x$ and so 
$x\gamma=v_{I_x}\le y\theta=y\gamma$. 
If $x\in A$ and $y\in X\setminus A$, then $x\in I_y$, whence 
$x\gamma=x\theta\le v_{I_y}=y\gamma$. Finally, suppose that $x,y\in X\setminus A$. 
Then, we have $I_x\subseteq I_y$. If $I_x=I_y$ then, trivially, $x\gamma=v_{I_x}=v_{I_y}=y\gamma$. If  $I_x\subsetneq I_y$ then we may take $a\in I_y\setminus I_x$ and we have $x\gamma=v_{I_x}\le a\theta\le v_{I_y}=y\gamma$. 
Thus, we proved that $\gamma\in\O(X)$.
Since, clearly, $\im(\gamma)\subseteq Y$, we have $\gamma\in\O(X,Y)$, 
as required. 
\end{proof}

Now, observe that, given $\alpha,\beta\in\O(X)$ such that $\alpha\mathscr{R}\beta$ in $\O(X)$, then $\alpha\mathscr{R}\beta$ in $\T(X)$ and so  $\ker(\alpha)=\ker(\beta)$.  
On the other hand, in \cite[Theorem 3.3]{Sanwong&Sommanee:2008} Sanwong and Sommanee showed that the relation $\mathcal{R}$ in $\T(X,Y)$ is just the restriction of the relation $\mathcal{R}$ in $\T(X)$, despite $\T(X,Y)$ may be non-regular.  

\medskip 

Let $\alpha,\beta\in\O(X)$ be such that $\ker(\alpha)=\ker(\beta)$. 
Define a relation $\theta:\im(\alpha)\longrightarrow \im(\beta)$ by 
$(a,b)\in\theta$ if and only if $a\alpha^{-1}=b\beta^{-1}$, 
for all $a\in\im(\alpha)$ and $b\in\im(\beta)$. It follows immediately from the equality 
$\ker(\alpha)=\ker(\beta)$ that $\theta$ is a bijective function. Moreover, $\theta$ is an order-preserving function. In fact, let $a_1,a_2\in\im(\alpha)$ be such that $a_1\le a_2$ and let $b_1=a_1\theta$ and $b_2=a_2\theta$. Take $x_1\in a_1\alpha^{-1}=b_1\beta^{-1}$ and 
$x_2\in a_2\alpha^{-1}=b_2\beta^{-1}$. If $x_2\le x_1$ then 
$a_2=x_2\alpha\le x_1\alpha=a_1$, whence $a_1=a_2$ and so $a_1\theta=a_2\theta$. 
On the other hand, if $x_1\le x_2$ then 
$a_1\theta=b_1=x_1\beta\le x_2\beta=b_2=a_2\theta$. Thus $\theta$ also preserves the order. 
We call to $\theta:\im(\alpha)\longrightarrow \im(\beta)$ the \textit{canonical order-isomorphism} 
associated to the pair $(\alpha,\beta)$. 
Notice that, given $x\in X$ and if $a=x\alpha$ and $b=x\beta$, 
we have $x\in a\alpha^{-1}\cap b\beta^{-1}$, from which it follows that $a\alpha^{-1}=b\beta^{-1}$ and so $x\alpha\theta=a\theta=b=x\beta$ and $x\beta\theta^{-1}=b\theta^{-1}=a=x\alpha$. Therefore, $\alpha=\beta\theta^{-1}$ and $\beta=\alpha\theta$. 

Observe that, if $\alpha=\beta$ then the canonical order-isomorphism 
$\theta:\im(\alpha)\longrightarrow \im(\beta)=\im(\alpha)$ is just the partial identity with domain $\im(\alpha)$. Notice also that any partial identity is always (bi)completable in $\O(X)$, but it may be not completable in $\O(X,Y)$. For instance, considering $\R$ equipped with the usual order, the partial identity with domain $\left]0,1\right[$ is, trivially, not completable in $\O(\R,\left]0,1\right[)$.

\begin{proposition}\label{ROXY}
Let $X$ be a chain and let $Y$ be a nonempty subset of $X$.  
Let $\alpha,\beta\in\O(X,Y)$. 
Then $\alpha\mathscr{R}\beta$ in $\O(X,Y)$ if and only if either $\alpha=\beta$ or $\ker(\alpha)=\ker(\beta)$ 
and the canonical order-isomorphism $\theta:\im(\alpha)\longrightarrow \im(\beta)$ is bicompletable in $\O(X,Y)$. 
Furthermore, if $\alpha$ and $\beta$ are regular elements of $\O(X,Y)$, then  
$\alpha\mathscr{R}\beta$ in $\O(X,Y)$ if and only if $\ker(\alpha)=\ker(\beta)$. 
\end{proposition}
\begin{proof}
First, suppose that $\alpha\mathscr{R}\beta$ in $\O(X,Y)$, with $\alpha\neq\beta$.  
Let $\gamma,\xi\in\O(X,Y)$ be such that 
$\alpha=\beta\gamma$ and $\beta=\alpha\xi$. As observed above, we have $\ker(\alpha)=\ker(\beta)$ and so we may consider 
the canonical order-isomorphism $\theta:\im(\alpha)\longrightarrow \im(\beta)$. 
Let $a\in\im(\alpha)$ and $b\in\im(\beta)$ be such that $a\alpha^{-1}=b\beta^{-1}$. 
Hence $a\theta=b$ (and $b\theta^{-1}=a$). Take $x\in a\alpha^{-1}=b\beta^{-1}$. 
Then $a\xi=x\alpha\xi=x\beta=b$ and $b\gamma=x\beta\gamma=x\alpha=a$. 
Therefore, we proved that $\xi$ and $\gamma$ are complete extensions in $\O(X,Y)$ of $\theta$ and $\theta^{-1}$, respectively. 

The converse is an immediate consequence of the equalities $\alpha=\beta\theta^{-1}$ and $\beta=\alpha\theta$. 

Finally, if $\alpha$ and $\beta$ are regular in $\O(X,Y)$, then $\alpha\mathscr{R}\beta$ in $\O(X,Y)$ if and only $\alpha\mathscr{R}\beta$ in $\T(X)$ and so the result follows, as required.  
\end{proof}

In view of the above observation, in particular, we have: 

\begin{corollary}\label{ROX}
Let $X$ be a chain and let $\alpha,\beta\in\O(X)$. 
Then $\alpha\mathscr{R}\beta$ in $\O(X)$ if and only if $\ker(\alpha)=\ker(\beta)$ 
and the canonical order-isomorphism $\theta:\im(\alpha)\longrightarrow \im(\beta)$ is bicompletable in $\O(X)$. 
\end{corollary}

We also have immediately: 

\begin{corollary}
Let $X$ be a chain, let $Y$ be a nonempty subset of $X$ and let $\alpha$ and $\beta$ be two regular elements of $\O(X,Y)$ such that $\ker(\alpha)=\ker(\beta)$. Then the canonical order-isomorphism $\theta:\im(\alpha)\longrightarrow \im(\beta)$ is bicompletable in $\O(X,Y)$. 
\end{corollary}

Observe that, given $\alpha,\beta\in\T(X,Y)$ with $\ker(\alpha)=\ker(\beta)$, 
it is clear (see \cite[Lemma 3.4]{Sanwong&Sommanee:2008}) that $\alpha\in\F(X,Y)$ if and only if $\beta\in\F(X,Y)$. This is, in fact, a trivial statement as, under this conditions, we have two $\mathscr{R}$-related elements and $\F(X,Y)=\reg(\T(X,Y))$. 

If $\alpha,\beta\in\O(X,Y)$ are such that $\ker(\alpha)=\ker(\beta)$, then from the previous property it follows immediately that also $\alpha\in\FO(X,Y)$ if and only if $\beta\in\FO(X,Y)$. However, notice that, in this case, we may not have  $\alpha\mathscr{R}\beta$ in $\O(X,Y)$ nor $\FO(X,Y)=\reg(\O(X,Y))$. 
Moreover, in fact, in $\O(X)$ is not true in general that the equality $\ker(\alpha)=\ker(\beta)$ suffices to imply that $\alpha\in\reg(O(X))$ if and only if $\beta\in\reg(O(X))$ (and, 
consequently, to imply that $\alpha$ and $\beta$ are $\mathscr{R}$-related). 

For instance, with $X=\R$ equipped with the usual order, being $\alpha\in\O(X)$ 
the exponential function and $\beta\in\O(X)$ the identity function, then $\alpha$ is non-regular (since $\im(\alpha)=\left]0,+\infty\right[$ has lower bounds in $\R$ but no minimum) and, contrariwise, $\beta$ is regular. However, we have $\ker(\alpha)=\ker(\beta)$. Obviously, $\alpha$ and $\beta$ are not $\mathscr{R}$-related, since one of them is regular and the other is not. 
Apart from this, notice that the canonical order-isomorphism 
$\theta:\im(\alpha)\longrightarrow \im(\beta)$ is the logarithm function, which 
is not completable in $\O(\R)$, as already observed.   

Observe also that two elements of $\O(X,Y)$ may be $\mathscr{R}$-related in $\O(X)$ but not 
$\mathscr{R}$-related in $\O(X,Y)$. For example, considering again $X=\R$ equipped with the usual order, let $\alpha,\beta\in\O(\R)$ be defined by $x\alpha=\arctan(x)$ and $x\beta=\arctan(x)+\frac{\pi}{2}$, for all $x\in\R$. 
Then, clearly, $\ker(\alpha)=\ker(\beta)$ (both are injective functions) and 
the canonical order-isomorphism 
$\theta:\im(\alpha)=\left]-\pi/2,\pi/2\right[\longrightarrow \im(\beta)=\left]0,\pi\right[$ (which is defined by $a\theta=a+\frac{\pi}{2}$, for 
$a\in \left]-\pi/2,\pi/2\right[)$ is bicompletable in $\O(\R)$. 
Thus $\alpha\mathscr{R}\beta$ in $\O(\R)$. On the other hand, let $Y=\left]-\pi/2,+\infty\right[$. Then, we also have $\alpha,\beta\in\O(\R,Y)$. 
However, despite $\theta$ is still completable in $\O(\R,Y)$, 
by the contrary, its inverse 
$\theta^{-1}:\im(\beta)=\left]0,\pi\right[\longrightarrow \im(\alpha)=\left]-\pi/2,\pi/2\right[$ is clearly not. 
Therefore $\alpha$ and $\beta$ are not $\mathscr{R}$-related in $\O(\R,Y)$. 

\bigskip 

As an immediate consequence of Propositions \ref{LOXY} and \ref{ROXY}, we have: 

\begin{corollary}\label{HOXY}
Let $X$ be a chain and let $Y$ be a nonempty subset of $X$.  
Let $\alpha,\beta\in\O(X,Y)$. 
Then $\alpha\mathscr{H}\beta$ in $\O(X,Y)$ if and only if either $\alpha=\beta$ or 
$\alpha,\beta\in\FO(X,Y)$, $\im(\alpha)=\im(\beta)$, 
$\ker(\alpha)=\ker(\beta)$ 
and the canonical order-isomorphism $\theta:\im(\alpha)\longrightarrow \im(\beta)$ is bicompletable in $\O(X,Y)$. 
Furthermore, if $\alpha$ and $\beta$ are regular elements of $\O(X,Y)$, then  
$\alpha\mathscr{H}\beta$ in $\O(X,Y)$ if and only if $\im(\alpha)=\im(\beta)$ and $\ker(\alpha)=\ker(\beta)$. 
\end{corollary}

If $Y = X$, like for the relations $\mathscr{L}$ and $\mathscr{R}$, a simpler statement can be presented:

\begin{corollary}
Let $X$ be a chain and let $\alpha,\beta\in\O(X)$. 
Then $\alpha\mathscr{H}\beta$ in $\O(X)$ if and only if $\im(\alpha)=\im(\beta)$,  $\ker(\alpha)=\ker(\beta)$ and the canonical order-isomorphism $\theta:\im(\alpha)\longrightarrow \im(\beta)$ is bicompletable in $\O(X)$. 
\end{corollary}

\medskip 

Now, recall that Sanwong and Sommanee showed in \cite[Theorem 3.7]{Sanwong&Sommanee:2008} that, 
given $\alpha,\beta\in\T(X,Y)$, we have $\alpha\mathscr{D}\beta$ in $\T(X,Y)$ if and only if 
either $\alpha,\beta\in\F(X,Y)$ and $|\im(\alpha)|=|\im(\beta)|$ or $\alpha,\beta\in\T(X,Y)\setminus\F(X,Y)$ and 
$\ker(\alpha)=\ker(\beta)$. For $\alpha,\beta\in\O(X,Y)$ such that $\alpha\mathscr{D}\beta$ in $\O(X,Y)$, 
since $\alpha\mathscr{D}\beta$ also in $\T(X,Y)$, we immediately deduce that 
either $\alpha,\beta\in\FO(X,Y)$ or $\alpha,\beta\in\O(X,Y)\setminus\FO(X,Y)$. 
Furthermore, in $\O(X,Y)$, we have: 

\begin{proposition}\label{DOXY}
Let $X$ be a chain, let $Y$ be a nonempty subset of $X$ and let $\alpha,\beta\in\O(X,Y)$. 
Then $\alpha\mathscr{D}\beta$ in $\O(X,Y)$ if and only if one of the following three statements holds: 
\begin{enumerate}
\item $\alpha\mathscr{L}\beta$ in $\O(X,Y)$; 
\item $\alpha,\beta\in\FO(X,Y)$
and there exists a bicompletable in $\O(X,Y)$ 
order-isomorphism $\theta:\im(\alpha)\rightarrow \im(\beta)$;
\item $\alpha,\beta\in\O(X,Y)\setminus\FO(X,Y)$
and $\alpha\mathscr{R}\beta$ in $\O(X,Y)$.
\end{enumerate}
Furthermore, if $\alpha$ and $\beta$ are regular elements of $\O(X,Y)$, then  
$\alpha\mathscr{D}\beta$ in $\O(X,Y)$ if and only if $\im(\alpha)$ and $\im(\beta)$ are order-isomorphic. 
\end{proposition}
\begin{proof}
First, suppose that $\alpha\mathscr{D}\beta$ and let $\gamma\in\O(X,Y)$ be such that 
$\alpha\mathscr{R}\gamma$ and $\gamma\mathscr{L}\beta$. Then, by Proposition \ref{ROXY}, 
either $\alpha=\gamma$ or 
$\ker(\alpha)=\ker(\gamma)$ and the canonical order-isomorphism 
$\theta:\im(\alpha)\longrightarrow \im(\gamma)$ is bicompletable in $\O(X,Y)$. 
On the other hand, by Proposition \ref{LOXY}, we have 
either $\gamma=\beta$ or $\gamma,\beta\in\FO(X,Y)$ and 
$\im(\gamma)=\im(\beta)$.

If $\alpha=\gamma$ then, trivially, $\alpha\mathscr{L}\beta$. Hence, 
let us suppose that $\ker(\alpha)=\ker(\gamma)$ and the canonical order-isomorphism 
$\theta:\im(\alpha)\longrightarrow \im(\gamma)$ is bicompletable in $\O(X,Y)$. 

As observed above, from the equality $\ker(\alpha)=\ker(\gamma)$ we obtain $\alpha\in\FO(X,Y)$ if and only if $\gamma\in\FO(X,Y)$. 

If $\gamma=\beta$ then either $\alpha,\beta\in\FO(X,Y)$ or $\alpha,\beta\in\O(X,Y)\setminus\FO(X,Y)$ and thus  
it follows immediately that statement 2 or 3 holds. 
On the other hand, suppose that $\gamma,\beta\in\FO(X,Y)$ and $\im(\gamma)=\im(\beta)$. Then, we have an  
order-isomorphism $\theta:\im(\alpha)\longrightarrow \im(\gamma)=\im(\beta)$ which is bicompletable in $\O(X,Y)$ 
and so statement 2 holds. 

Observe that, in any case, $\im(\alpha)$ and $\im(\beta)$ are order-isomorphic. 

In order to prove the converse implication, observe that
if either statement 1 or 3 holds then, trivially, we have $\alpha\mathscr{D}\beta$. 

Therefore, suppose that $\alpha,\beta\in\FO(X,Y)$ and there exists an order-isomorphism
$\theta:\im(\alpha)\longrightarrow \im(\beta)$. 

Define a transformation $\gamma$ of $X$ by $x\gamma=(x\alpha)\theta$, for all $x\in X$. 
Clearly, $\gamma\in\FO(X,Y)$ and $\im(\gamma)=\im(\beta)$ and so, 
by Proposition \ref{LOXY}, it follows that $\gamma\mathscr{L}\beta$. 
On the other hand, given $x,y\in X$, since $\theta$ is a bijection, we have $x\alpha=y\alpha$ if and only if $x\gamma=(x\alpha)\theta=(y\alpha)\theta=y\gamma$. 
Thus $\ker(\alpha)=\ker(\gamma)$. 

Now, in particular, if $\alpha$ and $\beta$ are regular elements of $\O(X,Y)$ then $\gamma$ is also a regular element 
of $\O(X,Y)$ (since $\gamma\mathscr{L}\beta$), whence from $\ker(\alpha)=\ker(\gamma)$ we obtain $\alpha\mathscr{R}\gamma$, by Proposition \ref{ROXY}. 

Next, take $a\in\im(\alpha)$ and $b\in\im(\gamma)$.  
If $a\alpha^{-1}=b\gamma^{-1}$ then, being $x\in a\alpha^{-1}=b\gamma^{-1}$, 
we have $a\theta=(x\alpha)\theta=x\gamma=b$. 
If $a\theta=b$ then, being $x\in a\alpha^{-1}$, we have $b=(x\alpha)\theta=x\gamma$, 
whence $x\in b\gamma^{-1}$ and so $x\in a\alpha^{-1}\cap b\gamma^{-1}$, 
from which it follows that $a\alpha^{-1}=b\gamma^{-1}$. 
Thus, the order-isomorphism $\theta:\im(\alpha)\longrightarrow \im(\beta)=\im(\gamma)$ is, 
in fact, the canonical order-isomorphism associated to $(\alpha,\gamma)$.  
Hence, in addition, if $\theta$ is bicompletable in $\O(X,Y)$ then, by Proposition \ref{ROXY}, we also have $\alpha\mathscr{R}\gamma$. 

Therefore, we proved that if either $\alpha$ and $\beta$ are regular elements of $\O(X,Y)$ with 
order-isomorphic images or there exists a bicompletable in $\O(X,Y)$ order-isomorphism
$\theta:\im(\alpha)\longrightarrow \im(\beta)$, then $\alpha\mathscr{D}\beta$, as required. 
\end{proof}

Notice that we may have elements $\alpha,\beta\in\FO(X,Y)$ such that $\im(\alpha)=\im(\beta)$, and so $\alpha\mathscr{L}\beta$ in $\O(X,Y)$, that do not verify the condition 2 above. 
Next, we provide such an example. 

Let $X=\R\setminus\{0\}$ equipped with the usual order and take $Y=X\setminus\{2\}=\R\setminus\{0,2\}$. 
Define $\alpha,\beta\in\T(X)$ by 
$$
x\alpha=\left\{
\begin{array}{ll}
1 & \mbox{if $x\le-1$}\\
x+2 & \mbox{if $-1<x<0$ or $0<x<1$}\\
3 & \mbox{if $1\le x\le2$}\\
x+1& \mbox{if $x>2$}
\end{array}
\right. 
\quad\text{and}\quad
x\beta=\left\{
\begin{array}{ll}
1 & \mbox{if $x\le-1$}\\
x+2 & \mbox{if $-1<x<0$ or $0<x<1$}\\
3 & \mbox{if $1\le x\le2$}\\
x^2-1& \mbox{if $x>2$.}
\end{array}
\right. 
$$
Clearly, $\alpha,\beta\in\O(X)$. Moreover, 
$\im(\alpha)=Y\alpha=\left[1,2\right[\cup\left]2,+\infty\right[\subseteq Y$ and 
$\im(\beta)=Y\beta=\left[1,2\right[\cup\left]2,+\infty\right[\subseteq Y$, 
whence $\alpha,\beta\in\FO(X,Y)$ and $\im(\alpha)=\im(\beta)$. 

On the other hand, suppose there exists an 
order-isomorphism $\theta:\left[1,2\right[\cup\left]2,+\infty\right[\longrightarrow \left[1,2\right[\cup\left]2,+\infty\right[$ which is completable in $\O(X,Y)$.  
Since $I=\left[1,2\right[$ is an order ideal of $A=\left[1,2\right[\cup\left]2,+\infty\right[$ 
and 
$\{x\in X\mid a<x<b, ~\mbox{for all $a\in I$ and $b\in A\setminus I$}\}=\{2\}\neq\emptyset$, 
by Proposition \ref{compl}, there exists an element $y\in Y$ such that 
$a\theta\le y\le b\theta$, for all $a\in I$ and $b\in A\setminus I$. 
Now, taking $a\in I$, we have $a\theta\ge1$, whence $y\ge1$ and so 
$y\in\left[1,2\right[\cup\left]2,+\infty\right[$. 
Therefore $y=c\theta$, for some $c\in\left[1,2\right[\cup\left]2,+\infty\right[$. 
From $a\theta\le c\theta\le b\theta$, for all $a\in I$ and $b\in A\setminus I$, it follows 
$a\le c\le b$, for all $a\in I$ and $b\in A\setminus I$, and so $c=2$, which is a contradiction. 
Therefore, there is no bicompletable in $\O(X,Y)$ 
order-isomorphism $\theta: \im(\alpha)\longrightarrow\im(\beta)$, as required. 

\medskip 

By taking into account that any partial identity is (bi)completable in $\O(X)$,  
we derive from Proposition \ref{DOXY} the following simpler statement for $Y = X$:

\begin{corollary}\label{DOX}
Let $X$ be a chain and let $\alpha,\beta\in\O(X)$. 
Then $\alpha\mathscr{D}\beta$ in $\O(X)$ if and only if there exists a bicompletable in $\O(X)$ 
order-isomorphism $\theta:\im(\alpha)\longrightarrow \im(\beta)$. 
\end{corollary}

\medskip 

Finally, we focus our attention in Green's relation $\mathscr{J}$. 
Regarding $\T(X,Y)$, Sanwong and Sommanee proved in \cite[Theorem 3.9]{Sanwong&Sommanee:2008} that, 
given $\alpha,\beta\in\T(X,Y)$, we have $\alpha\mathscr{J}\beta$ in $\T(X,Y)$ if and only if 
$|\im(\alpha)|=|Y\alpha|=|Y\beta|=|\im(\beta)|$ or $\ker(\alpha)=\ker(\beta)$. 

\smallskip

Observe that any injective order-preserving function $\theta:A\longrightarrow B$,  
with $A$ and $B$ two subchains of $X$, induces an order-isomorphism 
$\bar\theta:A\longrightarrow A\theta$. By convenience, to the inverse order-isomorphism 
${\bar\theta}^{-1}:A\theta\longrightarrow A$ we also call \textit{inverse} of $\theta$ 
and we simply denote ${\bar\theta}^{-1}$ by $\theta^{-1}$. 

\begin{proposition}\label{JOXY}
Let $X$ be a chain, let $Y$ be a nonempty subset of $X$ and let $\alpha,\beta\in\O(X,Y)$. 
Then $\alpha\mathscr{J}\beta$ in $\O(X,Y)$ if and only if one of the following three statements holds: 
\begin{enumerate}
\item $\alpha\mathscr{L}\beta$ in $\O(X,Y)$; 
\item $\alpha\mathscr{R}\beta$ in $\O(X,Y)$; 
\item There exist injective order-preserving functions $\theta:\im(\alpha)\longrightarrow Y\beta$ 
and $\tau:\im(\beta)\longrightarrow Y\alpha$ admitting completable inverses in $\O(X,Y)$.
\end{enumerate}
\end{proposition}
\begin{proof}
First, suppose that $\alpha\mathscr{J}\beta$. 
Let $\lambda,\gamma,\delta,\xi\in\O(X,Y)^1$ be such that $\alpha=\lambda\beta\gamma$ and 
$\beta=\delta\alpha\xi$. Since $\alpha=(\lambda\delta\lambda)\beta(\gamma\xi\gamma)$ and 
$\beta=(\delta\lambda\delta)\alpha(\xi\gamma\xi)$, we may assume, with no loss of generality, 
that $\lambda\in\O(X,Y)$ if and only if $\delta\in\O(X,Y)$ and also  
$\gamma\in\O(X,Y)$ if and only if $\xi\in\O(X,Y)$. Therefore, we have three cases to consider. 

If $\gamma=1$ then $\alpha=\lambda\beta$ and 
$\beta=\delta\alpha$ and so $\alpha\mathscr{L}\beta$.

Similarly, if $\lambda=1$ then $\alpha=\beta\gamma$ and 
$\beta=\alpha\xi$, whence $\alpha\mathscr{R}\beta$. 

Thus, it remains to consider the case $\lambda,\gamma,\delta,\xi\in\O(X,Y)$. 

For each $a\in\im(\alpha)$, we choose an element $w_a\in\im(\lambda\beta)\cap a\gamma^{-1}\subseteq Y\beta$. 
Then, define a map 
$\theta:\im(\alpha)\longrightarrow Y\beta$ by $a\theta=w_a$, 
for all $a\in\im(\alpha)$. Let $a,b\in\im(\alpha)$ be such that $a<b$. 
If $w_b\le w_a$, then $b=w_b\gamma\le w_a\gamma=a$, which is a contradiction. 
Thus $a\theta=w_a<w_b=b\theta$ and so $\theta$ is an injective order-preserving function. 
Now, since $w_a\theta^{-1}=a=w_a\gamma$, for all $a\in\im(\alpha)$, it follows that $\gamma$ is a complete extension of $\theta^{-1}:\im(\theta)\longrightarrow\im(\alpha)$ in $\O(X,Y)$.  
Similarly, we define an injective order-preserving function $\tau:\im(\beta)\longrightarrow Y\alpha$ admitting a completable inverse in $\O(X,Y)$.

Conversely, if either $\alpha\mathscr{L}\beta$ or $\alpha\mathscr{R}\beta$ then, trivially, $\alpha\mathscr{J}\beta$. 
So, assume that statement 3 holds. 

Let $\theta:\im(\alpha)\longrightarrow Y\beta$ be an injective order-preserving function such that its inverse $\theta^{-1}:\im(\theta)\longrightarrow \im(\alpha)$ has a complete extension $\gamma\in\O(X,Y)$. 
For each $b\in\im(\theta)\subseteq Y\beta$ choose an element $z_b\in b\beta^{-1}\cap Y$. 
Let $\lambda$ be the transformation of $X$ defined by $x\lambda=z_{x\alpha\theta}$, for all $x\in X$.  
Hence $\lambda\in\O(X,Y)$. In fact, take $x,y\in X$ such that $x\le y$. 
Then $x\alpha\le y\alpha$ and so $x\alpha\theta\le y\alpha\theta$. If 
$z_{x\alpha\theta}\ge z_{y\alpha\theta}$ then 
$x\alpha\theta=z_{x\alpha\theta}\beta\ge  z_{y\alpha\theta}\beta=y\alpha\theta$, 
whence $x\alpha\theta=y\alpha\theta$ and so  
$z_{x\alpha\theta}=z_{y\alpha\theta}$. Thus $z_{x\alpha\theta}\le z_{y\alpha\theta}$, i.e. 
$x\lambda\le y\lambda$. Moreover, for each $x\in X$, we have 
$
x\lambda\beta\gamma=z_{x\alpha\theta}\beta\gamma=x\alpha\theta\gamma=x\alpha\theta\theta^{-1}=x\alpha.
$
Therefore $\alpha=\lambda\beta\gamma$. 

Similarly, by supposing the existence of an injective order-preserving function 
$\tau:\im(\beta)\longrightarrow Y\alpha$ with completable inverse in $\O(X,Y)$, we may find elements  
$\delta,\xi\in\O(X,Y)$ such that $\beta=\delta\alpha\xi$. Thus, we proved that $\alpha\mathscr{J}\beta$, as required.  
\end{proof}

In view of Corollaries \ref{LOX} and \ref{ROX} and by taking into account once again that any partial identity is (bi)completable in $\O(X)$, for $Y = X$, we may state simply:

\begin{corollary}\label{JOX}
Let $X$ be a chain and let $\alpha,\beta\in\O(X)$. 
Then $\alpha\mathscr{J}\beta$ in $\O(X)$ if and only if there exist
injective order-preserving functions $\theta:\im(\alpha)\longrightarrow \im(\beta)$ 
and $\tau:\im(\beta)\longrightarrow \im(\alpha)$ admitting completable inverses in $\O(X)$. 
\end{corollary}

We finish this section by showing that we may have $\mathscr{D}\subsetneq\mathscr{J}$ in $\O(X)$, or even in $\O(X,Y)$, with $Y$ a proper nonempty subset of $X$. 

Let $X=\R$ equipped with the usual order. 
Let $\alpha,\beta\in\O(X)$ be defined by $x\alpha=\arctan(x)$, for $x\in\R$, and 
$$
x\beta=\left\{
\begin{array}{cl}
-1 & \mbox{if $x<-1$}\\
x & \mbox{if $-1\le x\le1$}\\
1& \mbox{if $x>1$.}
\end{array}
\right. 
$$
Then $\im(\alpha)=\left]-\pi/2,\pi/2\right[$ and $\im(\beta)=\left[-1,1\right]$, which are not order-isomorphic intervals of $\R$ and so $(\alpha,\beta)\not\in\mathscr{D}$ in $\O(\R)$. 
On the other hand, it is easy to show that $\alpha\mathscr{J}\beta$ in $\O(\R)$. In fact, given any two bounded intervals $I$ and $J$ of $\R$ we may find an injective order-preserving function $I\longrightarrow J$ admitting a completable inverse in $\O(\R)$.  
By taking, for instance, $Y=\left[-\pi/2,\pi/2\right]$, we still have  
$\alpha,\beta\in\O(X,Y)$ and it is also easy to show that $(\alpha,\beta)\not\in\mathscr{D}$ and $\alpha\mathscr{J}\beta$ in $\O(X,Y)$.

\section{An isomorphism theorem} \label{iso}

For finite chains $X$ and $X'$, it is easy to show that
the monoids $\O(X)$ and $\O(X')$ are isomorphic if and only if $|X|=|X'|$. 
In general, given finite or infinite chains $X$ and $X'$, 
it is well known, and in fact not difficult to prove (for completeness sake, an argument will be presented below), that 
the monoids $\O(X)$ and $\O(X')$ are isomorphic if and only if $X$ and $X'$ are order-isomorphic or order-anti-isomorphic. 
Notice that, if $X$ and $X'$ are finite chains, then $X$ and $X'$ are order-isomorphic if and only if $X$ and $X'$ are order-anti-isomorphic if and only if $|X|=|X'|$.  

From now on, we are mainly interested in the case where $X$ is a finite chain. However, since some arguments are valid in general, we only make such restriction whenever necessary. 

Being $X$ a set and $x\in X$, we denote by $\X_x$ the constant transformation of $\T(X)$ with image $\{x\}$. Observe that, given $x\in X$ and $\alpha\in\T(X)$, we have 
$\X_x\alpha=\X_{x\alpha}$ and $\alpha\X_x=\X_x$. These immediate equalities allow us to easily deduce the following properties.   

\begin{lemma}\label{basic}
Let $X$ and $X'$ be two chains and let $Y$ and $Y'$ be nonempty subsets of $X$ and $X'$, respectively. Let $\Theta:\O(X,Y)\longrightarrow\O(X',Y')$ be an isomorphism. Then:
\begin{enumerate}
\item For all $x\in Y$ there exists (a unique) $x'\in Y'$ such that 
$\X_x\Theta=\X'_{x'}$; 
\item $\Theta$ induces a bijection $\theta:Y\longrightarrow Y'$ defined by $\X_x\Theta=\X'_{x\theta}$, 
for all $x\in Y$; 
\item $(x\theta)(\alpha\Theta)=(x\alpha)\theta$, for all $x\in Y$ and $\alpha\in\O(X,Y)$; 
\item $\fix(\alpha\Theta)=(\fix(\alpha))\theta$, for all $\alpha\in\O(X,Y)$; 
in particular,  $\im(\alpha\Theta)=(\im(\alpha))\theta$, 
for any idempotent $\alpha\in\O(X,Y)$;
\item $\im(\alpha\Theta)=(\im(\alpha))\theta$, 
for all $\alpha\in\O(X,Y)$ such that $|\im(\alpha)|=2$; 
\item The bijection $\theta:Y\longrightarrow Y'$ is either an order-isomorphism or an order-anti-isomorphism. 
\end{enumerate}
\end{lemma}
\begin{proof}
Let $x\in Y$. Then, for all $\alpha\in\O(X,Y)$, we have $\alpha\X_x=\X_x$ and so $(\alpha\Theta)(\X_x\Theta)=\X_x\Theta$. Since $\Theta$ is surjective, it follows that $\beta(\X_x\Theta)=\X_x\Theta$, for $\beta\in\O(X',Y')$. Take any element $z\in Y'$. 
Then $\X'_{(z)(\X_x\Theta)}=\X'_z(\X_x\Theta)=\X_x\Theta$ and thus $\X_x\Theta=\X'_{x'}$, for some (unique, since $\Theta$ is a function) $x'\in Y'$. 

Therefore, we have a well defined function $\theta:Y\longrightarrow Y'$ satisfying the equality $\X_x\Theta=\X'_{x\theta}$, for all $x\in Y$. 
A similar reasoning applied to the inverse isomorphism 
$\Theta^{-1}:\O(X',Y')\longrightarrow\O(X,Y)$ allows us to show the existence of a function $\theta':Y'\longrightarrow Y$ satisfying the equality $\X'_{x'}\Theta^{-1}=\X_{x'\theta'}$, for all $x'\in Y'$. Now, we have $\X_x=\X_x\Theta\Theta^{-1}= \X'_{x\theta}\Theta^{-1}=\X_{x\theta\theta'}$, for all $x\in Y$, 
and similarly $\X'_{x'}=\X'_{x'\theta'\theta}$, for all $x'\in Y'$. Thus $\theta$ and $\theta'$ are mutually inverse bijections. 

Next, we prove property 3. Let $x\in Y$ and $\alpha\in\O(X,Y)$. 
Then $\X'_{(x\theta)(\alpha\Theta)}=\X'_{x\theta}(\alpha\Theta)=
(\X_x\Theta)(\alpha\Theta)=(\X_x\alpha)\Theta=\X_{x\alpha}\Theta=\X'_{(x\alpha)\theta}$ and so $(x\theta)(\alpha\Theta)=(x\alpha)\theta$. 

In order to prove 4, let us take $\alpha\in\O(X,Y)$. 
If $x'\in \fix(\alpha\Theta)$ then $x'=x\theta$, for some $x\in Y$, and 
$(x\alpha)\theta=(x\theta)(\alpha\Theta) = x'(\alpha\Theta)= x' = x\theta$, 
whence $x\alpha=x$, since $\theta$ is injective, and so $x'=x\theta\in(\fix(\alpha))\theta$.   
Conversely, if $x\in\fix(\alpha)$ then $(x\theta)(\alpha\Theta)=(x\alpha)\theta=x\theta$, 
i.e. $x\theta\in\fix(\alpha\Theta)$. Thus $\fix(\alpha\Theta)=(\fix(\alpha))\theta$. 

The second statement of 4 follows immediately from the fact that $\im(\alpha)=\fix(\alpha)$, for any idempotent $\alpha\in\T(X)$. 

Regarding 5, 
let $\alpha\in\O(X,Y)$ be such that $\im(\alpha)=\{a<b\}$ and 
define $\varepsilon\in\T(X)$ by $x\varepsilon=a$, if $x\le a$, and $x\varepsilon=b$, otherwise. Clearly, $\im(\varepsilon)=\im(\alpha)$, $\varepsilon^2=\varepsilon$ and $\varepsilon\in\O(X,Y)$. 
Moreover, $\alpha\varepsilon=\alpha$, whence 
$(\alpha\Theta)(\varepsilon\Theta)=\alpha\Theta$ and so 
$\im(\alpha\Theta)\subseteq\im(\varepsilon\Theta)=
(\im(\varepsilon))\theta=(\im(\alpha))\theta$. 
Now, since $\alpha$ is non-constant, then $\alpha\Theta$ is non-constant. 
Hence $|\im(\alpha\Theta)|\ge2=|(\im(\alpha))\theta|$ and thus $\im(\alpha\Theta)=(\im(\alpha))\theta$. 

Finally, we prove 6. We may suppose that $|Y|\ge1$ 
(in fact, we may even suppose that $|Y|\ge2$). 
Let $a,b,c,d\in Y$ be such that $a<b$ and $c<d$. Define $\alpha\in\T(X)$ by $x\alpha=c$, if $x\le a$, and $x\alpha=d$, otherwise. Clearly, $\alpha\in\O(X,Y)$. Moreover, 
$(a\theta)(\alpha\Theta)=(a\alpha)\theta=c\theta$ and    
$(b\theta)(\alpha\Theta)=(b\alpha)\theta=d\theta$, whence $a\theta<b\theta$ if and only if $c\theta<d\theta$, since $\alpha\Theta$ is order-preserving and $\theta$ is injective. Therefore, $\theta$ is either an order-isomorphism or an order-anti-isomorphism, as required. 
\end{proof}

Now, notice that, if $X$ and $X'$ are two chains, $\varphi:X\longrightarrow X'$ 
either an order-isomorphism or an order-anti-isomorphism
and $Y$ and $Y'$ nonempty subsets of $X$ and $X'$, respectively, such that 
$Y\varphi=Y'$, then it is a routine matter to show that the map $\Theta:\O(X,Y)\longrightarrow\O(X',Y')$ defined by $\alpha\Theta=\varphi^{-1}\alpha\varphi$, 
for all $\alpha\in\O(X,Y)$, is an isomorphism. By combining this fact together with property 5 of Lemma \ref{basic}, we immediately get the following well known result, already recalled: 

\begin{corollary}
Let $X$ and $X'$ be two chains. Then the monoids $\O(X)$ and $\O(X')$ are isomorphic if and only if $X$ and $X'$ are either order-isomorphic or order-anti-isomorphic. 
\end{corollary}

On the other hand, we show next that, for finite chains $X$ and $X'$ the converse of the 
aforementioned property also is valid for non-trivial subchains $Y$ and $Y'$ of $X$ and $X'$, respectively. 
Notice that, if $|Y|=|Y'|=1$ then the semigroups $\O(X,Y)$ and $\O(X',Y')$ are always trivial (even with $X$ or $X'$ infinite) and so isomorphic.  

Observe also that, if $X$ and $X'$ are two finite chains with the same size, then there exists a unique order-isomorphism $\iota:X\longrightarrow X'$ and a unique order-anti-isomorphism $\sigma:X\longrightarrow X'$. Furthermore, if $X=\{x_1<x_2<\cdots<x_n\}$ and $X'=\{x'_1<x'_2<\cdots<x'_n\}$, for some $n\in\N$, then $x_i\iota=x'_i$ and 
$x_i\sigma=x'_{n-i+1}$, for all $i\in\{1,2,\ldots,n\}$. 

\begin{theorem}\label{finiteiso}
Let $X$ and $X'$ be two finite chains and let $Y$ and $Y'$ be nonempty subsets of $X$ and $X'$, respectively. Then the semigroups $\O(X,Y)$ and $\O(X',Y')$ are isomorphic if and only if one of the following conditions holds: 
\begin{enumerate}
\item $|Y|=|Y'|=1$; 
\item $|X|=|X'|$ and $Y\iota=Y'$, where $\iota:X\longrightarrow X'$ is the (unique) order-isomorphism; 
\item $|X|=|X'|$ and $Y\sigma=Y'$, where $\sigma:X\longrightarrow X'$ is the (unique) order-anti-isomorphism.  
\end{enumerate}
\end{theorem}
\begin{proof}
Based on the above, if either condition 1, 2 or 3 is satisfied then the semigroups $\O(X,Y)$ and $\O(X',Y')$ are isomorphic. Therefore, conversely assume there exists an isomorphism $\Theta:\O(X,Y)\longrightarrow\O(X',Y')$ and let $\theta:Y\longrightarrow Y'$ be the order-isomorphism or order-anti-isomorphism induced by $\Theta$ as given by Lemma \ref{basic}. 

Take $X=\{x_1<x_2<\cdots<x_m\}$ and $X'=\{x'_1<x'_2<\cdots<x'_n\}$, for some $m,n\in\N$,
and $Y=\{x_{i_1}<x_{i_2}<\cdots<x_{i_k}\}$ and $Y'=\{x'_{j_1}<x'_{j_2}<\cdots<x'_{j_k}\}$, for some $1\le k\le\min\{m,n\}$. If $k=1$ then condition 1 holds. Hence, from now on, we suppose that $k\ge2$. 

Next, notice that $x_{i_t}\theta=x'_{j_t}$, for all $t\in\{1,2,\ldots,k\}$, if $\theta$ is an order-isomorphism, and $x_{i_t}\theta=x'_{j_{k-t+1}}$, for all $t\in\{1,2,\ldots,k\}$, otherwise. 

Let $2\le t\le k$. Let 
$$
A_1=\{\alpha\in\O(X,Y)\mid \alpha=\alpha^2\;\text{and}\;\im(\alpha)=\{x_{i_1}<x_{i_t}\}\}
$$
and
$$ 
B_1=\{\beta\in\O(X',Y')\mid \beta=\beta^2\;\text{and}\;\im(\beta)=\{x_{i_1}\theta, x_{i_t}\theta\}\}.
$$
Then, by Lemma \ref{basic}, $A_1\Theta=B_1$. Moreover, we have 
\begin{equation}\label{ab1}
|A_1|=i_t-i_1
\quad\text{and}\quad 
|B_1|=\left\{\begin{array}{ll}
     j_t-j_1 & \mbox{if $\theta$ is an order-isomorphism}\\
     j_k-j_{k-t+1} & \mbox{otherwise}. 
    \end{array}\right.
\end{equation}
Next, let 
$$
A_2=\{\alpha\in\O(X,Y)\mid \alpha\neq\alpha^2,\; 
\im(\alpha)=\{x_{i_1}<x_{i_k}\}
\;\text{and}\;
x_{i_k}\in\fix(\alpha)\}
$$
and
$$ 
B_2=\{\beta\in\O(X',Y')\mid \beta\neq\beta^2,\;
\im(\beta)=\{x_{i_1}\theta, x_{i_k}\theta\}
\;\text{and}\;
x_{i_k}\theta\in\fix(\beta)\}.
$$
Again by Lemma \ref{basic}, we get $A_2\Theta=B_2$. Regarding the sizes, we have 
\begin{equation}\label{ab2}
|A_2|=i_1-1
\quad\text{and}\quad 
|B_2|=\left\{\begin{array}{ll}
     j_1-1 & \mbox{if $\theta$ is an order-isomorphism}\\
     n-j_k & \mbox{otherwise}. 
    \end{array}\right.
\end{equation}
Finally, let 
$$
A_3=\{\alpha\in\O(X,Y)\mid \alpha\neq\alpha^2,\; 
\im(\alpha)=\{x_{i_1}<x_{i_k}\}
\;\text{and}\;
x_{i_1}\in\fix(\alpha)\}
$$
and
$$ 
B_3=\{\beta\in\O(X',Y')\mid \beta\neq\beta^2,\;
\im(\beta)=\{x_{i_1}\theta, x_{i_k}\theta\}
\;\text{and}\;
x_{i_1}\theta\in\fix(\beta)\}.
$$
Once again by Lemma \ref{basic}, we obtain $A_3\Theta=B_3$. In this case, we have  
\begin{equation}\label{ab3}
|A_3|=m-i_k
\quad\text{and}\quad 
|B_3|=\left\{\begin{array}{ll}
     n-j_k & \mbox{if $\theta$ is an order-isomorphism}\\
     j_1-1 & \mbox{otherwise}. 
    \end{array}\right.
\end{equation}

Now, we analyze the equalities $|A_1|=|B_1|$, $|A_2|=|B_2|$ and $|A_3|=|B_3|$,
by considering two cases. 

First, suppose that $\theta$ is an order-isomorphism. Then, by (\ref{ab2}), we have 
$i_1=j_1$. By (\ref{ab1}) follows that $i_t=j_t$,  for $2\le t\le k$, and so  
by using (\ref{ab3}) we also deduce that $m=n$. Thus, in this case, condition 2 holds. 

Finally, we suppose that $\theta$ is an order-anti-isomorphism. 
Therefore, by (\ref{ab2}), we have $j_k=n-i_1+1$. Next, by (\ref{ab1}) and (\ref{ab2}), 
it follows that $j_{k-t+1}=j_k-i_t+i_1=(n-i_1+1)-i_t+i_1=n-i_t+1$, for $2\le t\le k$. 
In particular, $j_1=n-i_k+1$, from which follows, by using (\ref{ab3}), that 
$n=(j_1-1)+i_k=(m-i_k)+i_k=m$. Thus, we have $m=n$ and 
$j_{k-t+1}=n-i_t+1$, for $1\le t\le k$, and so, in this case, condition 3 holds, as required. 
\end{proof}

\smallskip 

Observe that, if $X$ is a finite chain, it is clear that the number of order-preserving mappings from $X$ into $Y$ coincides with   
the number of combinations of $|Y|$ objects
taken $|X|$ at a time, repetitions being permitted, 
i.e. 
$$
|\O(X,Y)|=\binom{|X|+|Y|-1}{|Y|-1}
$$ 
(see \cite{Hall:1967}). 

\smallskip 

From the above results, in order to study the semigroups with restricted range $\O(X,Y)$, with $X$ a finite chain, it suffices to consider the semigroups $\O_n(Y)$, with $Y$ a subchain of $\{1<2<\cdots<n\}$ and $n\in\N$. 
Let us denote by $\sigma$ the permutation that reflects $\{1<2<\cdots<n\}$ (the unique order-anti-automorphism of $\{1<2<\cdots<n\}$), i.e. 
$$
\sigma =
     \left(
       \begin{array}{ccccc}
                1 & 2 & \cdots & n-1 & n\\
               n & n-1 & \cdots & 2 & 1\\
       \end{array}
     \right).
$$ 
Following along this line, we may rewrite Theorem \ref{finiteiso}. 

\begin{corollary}
Let $n\in\N$ and let $Y$ and $Z$ be nonempty subsets of $\{1,2,\ldots,n\}$.  
Then the subsemigroups $\O_n(Y)$ and $\O_n(Z)$ of $\O_n$ are isomorphic if and only if 
$|Y|=|Z|=1$ or $Y=Z$ or $Y\sigma=Z$.  
\end{corollary}

For infinite chains $X$ and $X'$, a result similar to Theorem \ref{finiteiso} is not true in general, i.e. we may have isomorphic semigroups $\O(X,Y)$ and $\O(X',Y')$, with $Y$ and $Y'$ non-trivial subchains of $X$ and $X'$, respectively, without $X$ and $X'$ being either order-isomorphic or order-anti-isomorphic (although $Y$ and $Y'$ must be either order-isomorphic or order-anti-isomorphic, by Lemma \ref{basic}). An example with $|Y|=|Y'|=2$ is presented below. 

We notice that Jitjankarn claims in the paper (preprint) \cite{Jitjankarn:2012} that, for $|Y|\ge5$, the semigroups $\O(X,Y)$ and $\O(X',Y')$ are isomorphic if and only if there exists either an order-isomorphism or an order-anti-isomorphism 
$\varphi:X\longrightarrow X'$ such that $Y\varphi=Y'$. This result was obtained independently and is \textit{almost} a more general result than Theorem \ref{finiteiso}. Its proof is also, naturally, longer and much more elaborate than ours. 
Despite that, the cases $|Y|=|Y'|=3$ and $|Y|=|Y'|=4$ (for $X$ and $X'$ infinite) seem to remain an open problem. 

\begin{example}\em 
Consider the chain with $2$ elements 
$
\U=\{\overline1<\overline2\} 
$
and the chains 
$
\Nn=\N_0\oplus\U
$
(where $\N_0$ is equipped with the usual order and $m < \overline 1$, for all $m\in\N_0$) and 
$
\Zn=\Z\oplus\U
$
(where $\Z$ is equipped with the usual order and $m < \overline 1$, for all $m\in\Z$). 
Notice that a typical element of $\O(\Nn,\U)$ is of the form
\begin{equation}\label{typicN}
\alpha_n=\left(
\begin{array}{c|c|cc}
0\,\cdots\, n & n+1\,\cdots\, {+\infty} & \overline1 & \overline2 \\
\overline1 & \overline2 & \overline2 & \overline2 
\end{array}
\right) 
\quad\text{or}\quad 
\alpha_{i,j,k}=\left(
\begin{array}{c|cc}
0\,\cdots\, {+\infty} & \overline1 & \overline2 \\
\overline{i}  & \overline{j} & \overline{k}
\end{array}
\right), 
\end{equation}
with $n\in\N_0$ and $1\le i\le j\le k\le 2$, and 
a typical element of $\O(\Zn,\U)$ is of the form
\begin{equation}\label{typicZ}
\beta_m=\left(
\begin{array}{c|c|cc}
{-\infty}\,\cdots\, m & m+1\,\cdots\, {+\infty} & \overline1 & \overline2 \\
\overline1 & \overline2 & \overline2 & \overline2 
\end{array}
\right) 
\quad\text{or}\quad 
\beta_{i,j,k}=\left(
\begin{array}{c|cc}
{-\infty}\,\cdots\, {+\infty} & \overline1 & \overline2 \\
\overline{i}  & \overline{j} & \overline{k}
\end{array}
\right), 
\end{equation} 
with $m\in\Z$ and $1\le i\le j\le k\le 2$. 

Consider the map $f:\N_0\longrightarrow\Z$ defined by 
$$
f(n)=\left\{\begin{array}{ll}
             \frac{n+1}{2} & \mbox{if $n$ is odd}\\
             -\frac{n}{2} & \mbox{otherwise,}
            \end{array}\right.
$$
which is, clearly, a bijection. 

Next, let $\Theta:\O(\Nn,\U)\longrightarrow\O(\Zn,\U)$ be the map 
defined by 
$$
\alpha_n\Theta=\beta_{f(n)} 
\quad\text{and}\quad
\alpha_{i,j,k}\Theta=\beta_{i,j,k}~, 
$$
for all $n\in\N_0$ and $1\le i\le j\le k\le 2$.
Clearly, $\Theta$ is bijective. Moreover, it is also a routine matter to show that $\Theta$ is a homomorphism. 

Therefore, $\O(\Nn,\U)$ and $\O(\Zn,\U)$ are isomorphic semigroups, despite $\Nn$ and $\Zn$ are neither order-isomorphic nor order-anti-isomorphic (since $\Nn$ has both maximum and minimum elements, while $\Zn$ only has maximum element). 
\end{example}

\section{On the semigroups $\O_n(Y)$}\label{rank}

The main objective of this section is to determine the ranks of the semigroups $\O_n(Y)$,  for all nonempty subset $Y$ of $\{1,2,\ldots,n\}$. 
Recall that the rank of a finite semigroup  
is the cardinality of a least-size generating set.  

We begin by presenting some basic structural properties. 

It is a well known fact that $\O_n$ is a regular semigroup \cite{Gomes&Howie:1992}. Therefore, as a particular instance of Mora and Kemprasit's results \cite[Theorems 3.1 and 3.6]{Mora&Kemprasit:2010}, we immediately have:  

\begin{theorem}
Let $Y$ be a nonempty subset of $\{1,2,\ldots,n\}$. 
Then 
$
\reg(\O_n(Y))=\{\alpha\in\O_n(Y)\mid \im(\alpha)=Y\alpha\}.
$
Moreover, $\O_n(Y)$ is a regular semigroup if and only if $Y=\{1,2,\ldots,n\}$ or $|Y|=1$ or $Y=\{1,n\}$. 
\end{theorem}

Next, notice that, as $\O_n$ is $\mathscr{H}$-trivial \cite{Gomes&Howie:1992}, then  $\O_n(Y)$ is also $\mathscr{H}$-trivial. 
Regarding the remaining Green's relations on $\O_n(Y)$, 
it is easy to show that Propositions \ref{LOXY}, \ref{ROXY} and \ref{DOXY} may be rephrased as follows: 

\begin{theorem}
Let $Y$ be a nonempty subset of $\{1,2,\ldots,n\}$. 
Let $\alpha,\beta\in\O_n(Y)$. Then: 
\begin{enumerate}
\item $\alpha\mathscr{L}\beta$ in  $\O_n(Y)$ if and only if either $\alpha=\beta$ or $\alpha,\beta\in \reg(\O_n(Y))$ and $\im(\alpha)=\im(\beta)$; 
\item $\alpha\mathscr{R}\beta$ in  $\O_n(Y)$ if and only if $\ker(\alpha)=\ker(\beta)$; 
\item $\alpha\mathscr{D}\beta$ in  $\O_n(Y)$ if and only if either (i) $\alpha,\beta\in \reg(\O_n(Y))$ and $|\im(\alpha)|=|\im(\beta)|$ 
or (ii) $\alpha,\beta\in \O_n(Y)\setminus\reg(\O_n(Y))$ and $\ker(\alpha)=\ker(\beta)$.
\end{enumerate}
\end{theorem}

Observe that, trivially, in $\O_n(Y)$ we have $\mathscr{D}=\mathscr{J}$ (since it is finite).   

\medskip 

Let $Y$ be a nonempty subset of $\{1,2,\ldots,n\}$. If $|Y|=1$ then $|\O_n(Y)|=1$ and so its rank is, trivially, equal to $1$. In the antipodes, if $|Y|=n$ then $\O_n(Y)=\O_n$, which rank (as a monoid) and size are well known to be respectively $n$ and $\binom{2n-1}{n-1}$ \cite{Gomes&Howie:1992}. 
Therefore, from now on, we suppose that $1<|Y|<n$ and take $r=|Y|$. 
Recall that we have $|\O_n(Y)|=\binom{n+r-1}{r-1}$. 

\medskip 

We say that an element $y\in Y$ is \textit{captive} if either $y\in\{1,n\}$ or $1<y<n$ and $y-1,y+1\in Y$. Denote by $Y^\sharp$ the subset of captive elements of $Y$.  

For instance, with $n=7$, we have $\{1,3,4,5\}^\sharp=\{1,4\}$, $\{2,3,4,5\}^\sharp=\{3,4\}$, $\{2,4,5,7\}^\sharp=\{7\}$, $\{1,7\}^\sharp=\{1,7\}$, 
$\{2,4,6\}^\sharp= \emptyset=\{2,3,5,6\}^\sharp$. 

This notion allows us to state our main result of this section. 

\begin{theorem}\label{main}
Let $1<r<n$ and let $Y$ be a subset of $\{1,2,\ldots,n\}$ with $r$ elements.     
Then  $\rank(\O_n(Y)) = \binom{n-1}{r-1} + |Y^\sharp|$.
\end{theorem}

The rest of this section (and paper) is dedicated to proving Theorem \ref{main}. 

In what follows, it will be convenient to fix two particular complete extensions in $\O_n$ of an order-preserving function between to subchains of $\{1<2<\cdots<n\}$.  
Take a partial (order-preserving) transformation 
$
\theta =
            \left(
              \begin{array}{ccc}
                a_{1} & \cdots & a_{k}\\
                b_{1} & \cdots & b_{k}\\
              \end{array}
            \right), 
$
with $1\le a_{1} < \cdots < a_{k}\le n$, 
$1\le b_{1} \le \cdots \le b_{k}\le n$ 
and $1 \le k \leq n$. 
We define the \textit{canonical complete extensions}
$\widehat{\theta}, \widetilde{\theta} \in \mathcal{O}_{n}$ of $\theta$ by
$$
x\widehat{\theta} =\left\{
            \begin{array}{ll}
                b_{1} &\mbox{if $1 \le x < a_{2}$}\\
                b_{j} &\mbox{if $a_{j} \le x < a_{j+1}$, for $2 \le j \le k-1$}\\
                b_{k} &\mbox{if $a_{k} \le x \le n$} 
            \end{array}\right. 
$$
and 
$$ 
x\widetilde{\theta} =\left\{
            \begin{array}{ll}
                b_{1} &\mbox{if $1 \le x \le a_{1}$}\\
                b_{j} &\mbox{if $a_{j-1} < x \le a_{j}$, for  $2 \le j \le k-1$}\\
                b_{k} &\mbox{if $a_{k-1} < x \le n$}~.
            \end{array}\right. 
$$
Observe that $\im(\widehat{\theta})=\im(\widetilde{\theta})=\im(\theta)$ and so, in particular, if $\im(\theta)\subseteq Y$ then $\widehat{\theta}, \widetilde{\theta} \in \mathcal{O}_{n}(Y)$. 

\smallskip 

For instance, let  
 $\theta=
 \left(\begin{array}{cccc}
 2&5&6&8\\
 1&3&5&7
 \end{array}\right). 
$
Then, in $\O_9$ we have 
$$
\widehat{\theta}=
 \left(\begin{array}{ccccccccc}
 1&\mathit2&3&4&\mathit5&\mathit6&7&\mathit8&9\\
 1&\mathit1&1&1&\mathit3&\mathit5&5&\mathit7&7
 \end{array}\right) 
\quad\textrm{and}\quad 
\widetilde{\theta}=
 \left(\begin{array}{ccccccccc}
1&\mathit2&3&4&\mathit5&\mathit6&7&\mathit8&9\\
1&\mathit1&3&3&\mathit3&\mathit5&7&\mathit7&7
 \end{array}\right). 
$$

\smallskip 

Let us consider $Y = \{y_{1} < \cdots < y_{r}\}$. 
The next two lemmas will provide us a set of generators of $\O_n(Y)$ containing only transformations of rank no less than $r-1$. 

\begin{lemma}\label{max1}
        Let $\alpha \in \mathcal{O}_n(Y)$ be such that $|\im(\alpha)| = r-1$.
        Then $\alpha = \beta\gamma$, for some $\beta, \gamma \in \mathcal{O}_n(Y)$ such that  $|\im(\beta)| = r$,
        $|\im(\gamma)| = r-1$ and $\gamma$ is regular.
\end{lemma}
\begin{proof}
Let  $j \in \{1, \ldots, r\}$ be such that $\im(\alpha) = Y \setminus \{y_{j}\}$. Then, put 
$$
\alpha =
        \left(
       \begin{array}{ccccccc}
                A_{1} & \cdots & A_{j-1} & A_{j} & A_{j+1} & \cdots  & A_{r-1}\\
                y_{1} & \cdots & y_{j-1}  & y_{j+1} & y_{j+2} & \cdots  & y_{r}\\
       \end{array}
     \right)
$$
and take  $k \in \{1, \ldots, r-1\}$ such that   $|A_{k}| \geq 2$. 
     If $j < k$, then define
     $$
     \beta =\left(
          \begin{array}{ccccccccccc}
               A_{1} & \cdots & A_{j} & A_{j+1} & \cdots & A_{k-1} & \min(A_k) & A_k\setminus\{\min(A_k)\} & A_{k+1} & \cdots & A_{r-1}\\
               y_{1} & \cdots & y_{j} & y_{j+1} & \cdots & y_{k-1} & y_{k} & y_{k+1} & y_{k+2}  & \cdots & y_{r}\\
            \end{array}
        \right)
        $$
     and
      $$
      \theta =\left(
          \begin{array}{cccccccccccc}
               y_{1} & \cdots & y_{j-1} & y_{j} & y_{j+1} & \cdots & y_{k-1} & y_{k} & y_{k+2} & \cdots & y_{r}\\
               y_{1} & \cdots & y_{j-1} & y_{j+1} & y_{j+2} & \cdots & y_{k} & y_{k+1} & y_{k+2}  & \cdots & y_{r}\\
            \end{array}
        \right)~.
      $$ 
    If $k< j$, then define
$$
    \beta =\left(
          \begin{array}{ccccccccccc}
               A_{1} & \cdots & A_{k-1} & \min(A_k) & A_k\setminus\{\min(A_k)\} & A_{k+1} & \cdots & A_{j-1} & A_{j} & \cdots & A_{r-1}\\
               y_{1} & \cdots & y_{k-1} & y_{k} & y_{k+1} & y_{k+2} & \cdots & y_{j} & y_{j+1} & \cdots & y_{r}\\
            \end{array}
        \right)
$$
     and
$$
      \theta =\left(
          \begin{array}{ccccccccccc}
               y_{1} & \cdots & y_{k-1} & y_{k} & y_{k+2} & \cdots & y_{j} & y_{j+1}  & \cdots & y_{r}\\
               y_{1} & \cdots & y_{k-1} & y_{k} & y_{k+1} & \cdots & y_{j-1} & y_{j+1} & \cdots & y_{r}\\
            \end{array}
        \right)~.
$$
Finally, if $j = k$, then define
$$\beta =
        \left(
          \begin{array}{cccccccc}
            A_{1} & \cdots & A_{j-1} & \min(A_j) & A_j\setminus\{\min(A_j)\}& A_{j+1} & \cdots & A_{r-1}\\
            y_{1} & \cdots & y_{j-1} & y_{j} & y_{j+1} & y_{j+2} & \cdots & y_{r}\\
          \end{array}
        \right)
$$ 
and
$$
\theta =\left(
          \begin{array}{ccccccccccc}
               y_{1} & \cdots & y_{j-1} & y_{j} & y_{j+2}  & \cdots & y_{r}\\
               y_{1} & \cdots & y_{j-1} & y_{j+1} & y_{j+2} & \cdots & y_{r}\\
            \end{array}
        \right)~.
$$
In all cases, we have $\beta, \widehat{\theta} \in \O_n(Y)$, with $|\im(\beta)| = r$ and $|\im(\widehat{\theta})| = r-1$.
Moreover, $\im(\widehat{\theta})=Y\widehat{\theta}$, whence $\widehat{\theta}$ is regular, and  it is a routine matter to verify that  
$\alpha = \beta\widehat{\theta}$, as required.
\end{proof}

\begin{lemma}\label{max2}
        Let $\alpha \in \O_n(Y)$ be such that $|\im(\alpha)| = k < r-1$.
        Then $\alpha = \beta\gamma$, for some $\beta, \gamma \in \O_n(Y)$ such that 
        $|\im(\beta)| = |\im(\gamma)|= k+1$.
\end{lemma}
\begin{proof} Take 
        $\alpha =
        \left(
       \begin{array}{ccccccc}
                A_{1} & A_{2} & \cdots  & A_{k}\\
                a_{1} & a_{2} & \cdots  & a_{k}\\
       \end{array}
     \right),$
     where $\{a_{1} < \cdots < a_{k}\} \subseteq Y$.
Since $k \leq r-2$, then there exist $u, v \in Y \setminus \{a_{1}, \ldots, a_{k}\}$ such that 
$a_1<\cdots<a_\ell  < u  < a_{\ell+1}<\cdots<a_{m}  < v < a_{m+1}<\cdots<a_k$, for some $0\le\ell\le m\le k$. 
On the other hand, since $k<n$, then there exists $j \in \{1, \ldots, k\}$ such that $|A_{j}| \geq 2$. 
Next, we consider three cases. 

\smallskip 

\noindent{\sc case 1.} First, suppose that $1\le j\le\ell$. 
Define a full transformation $\beta$ by 
$$
\beta = \left\{ \begin{array}{l}
\left(
\begin{array}{cccccccc}
A_1 &\cdots & A_{\ell-1} & \min(A_\ell) & A_j\setminus\{\min(A_\ell)\} & A_{\ell+1} & \cdots & A_{k}\\
a_1& \cdots & a_{\ell-1}& a_{\ell} & u & a_{\ell+1} &\cdots & a_{k}
            \end{array}
\right)  \\ \qquad\mbox{if $j=\ell$}
\\ \\
\left(
    \begin{array}{cccccccccccc}
A_1 &\cdots & A_{j-1} & \min(A_j) & A_j\setminus\{\min(A_j)\} & A_{j+1} & \cdots & A_{\ell-1} & A_{\ell} & A_{\ell+1} & \cdots & A_{k}\\
a_1& \cdots & a_{j-1}& a_{j} & a_{j+1} & a_{j+2} & \cdots & a_{\ell} & u & a_{\ell+1} & \cdots & a_{k}
            \end{array}
\right) \\ \qquad\mbox{otherwise} 
\end{array} \right.
$$
and a partial transformation $\theta$ by 
$$
\theta = \left\{ \begin{array}{l}
\left(
\begin{array}{ccccccc}
a_1 & \cdots & a_{m}   & v & a_{m+1} & \cdots & a_k\\
a_1& \cdots  & a_{m}   & v & a_{m+1} & \cdots & a_k\   
         \end{array}
\right)  \\ \qquad \mbox{if $j=\ell$}
\\ \\
\left(
    \begin{array}{cccccccccccccc}
a_1 &\cdots & a_j & a_{j+2} & \cdots & a_{\ell}    & u            & a_{\ell+1} & \cdots & a_{m}   & v & a_{m+1} & \cdots & a_k\\
a_1& \cdots & a_j & a_{j+1} & \cdots & a_{\ell-1} & a_{\ell} & a_{\ell+1} & \cdots & a_{m}   & v & a_{m+1} & \cdots & a_k             
\end{array}
\right) \\ \qquad \mbox{otherwise.} 
\end{array} \right.
$$

\noindent{\sc case 2.} Next, suppose that $\ell+1\le j\le m$. Now, define a full transformation $\beta$ by 
$$
\beta = \left\{ \begin{array}{l}
\left(
\begin{array}{cccccccc}
A_1 &\cdots & A_{\ell} & \min(A_{\ell+1}) & A_{\ell+1}\setminus\{\min(A_{\ell+1})\} & A_{\ell+2} & \cdots & A_{k}\\
a_1& \cdots & a_{\ell}& u & a_{\ell+1} & a_{\ell+2} &\cdots & a_{k}
            \end{array}
\right)  \\ \qquad \mbox{if $j=\ell+1$}
\\ \\
\left(
    \begin{array}{cccccccccccc}
A_1 &\cdots & A_{\ell}  & A_{\ell+1}  &  A_{\ell+2} & \cdots & A_{j-1}   & \min(A_j) & A_j\setminus\{\min(A_j)\} & A_{j+1} & \cdots & A_{k}\\
a_1& \cdots & a_{\ell}  & u  & a_{\ell+1}  & \cdots & a_{j-2}   & a_{j-1} & a_{j} & a_{j+1} & \cdots & a_{k}
            \end{array}
\right) \\ \qquad \mbox{otherwise} 
\end{array} \right.
$$
and a partial transformation $\theta$ by 
$$
\theta = \left\{ \begin{array}{l}
\left(
\begin{array}{ccccccccccc}
a_1 & \cdots & a_\ell &  u & a_{\ell+2} & \cdots &  a_{m}   & v & a_{m+1} & \cdots & a_k\\
a_1& \cdots  & a_\ell &  a_{\ell+1} & a_{\ell+2} & \cdots &  a_{m}   & v & a_{m+1} & \cdots & a_k\   
         \end{array}
\right)  \\ \qquad \mbox{if $j=\ell+1$}
\\ \\
\left(
    \begin{array}{cccccccccccccc}
a_1 &\cdots &  a_\ell & u             & a_{\ell+1} & \cdots &  a_{j-1}  & a_{j+1} & \cdots & a_{m}   & v & a_{m+1} & \cdots & a_k\\
a_1 & \cdots & a_\ell & a_{\ell+1} & a_{\ell+2} & \cdots &  a_{j}     & a_{j+1} & \cdots & a_{m}   & v & a_{m+1} & \cdots & a_k             
\end{array}
\right) \\ \qquad \mbox{otherwise.} 
\end{array} \right.
$$
       
\noindent{\sc case 3.} Finally, suppose that $m+1\le j\le k$. In this last case, we define a full transformation $\beta$ by 
$$
\beta=\left\{ \begin{array}{l}
\left(
          \begin{array}{cccccccccccc}
 A_1 & \cdots & A_m &  \min(A_{m+1}) & A_{m+1}\setminus\{\min(A_{m+1})\} & A_{m+2} & \cdots &  A_{k}\\
 a_1 & \cdots & a_m & v & a_{m+1} & a_{m+2} & \cdots &  a_{k}\\
            \end{array}
\right) \\ \qquad \mbox{if $j=m+1$}
\\ \\
\left(
          \begin{array}{cccccccccccc}
 A_1 & \cdots & A_m & A_{m+1} & A_{m+2} & \cdots & A_{j-1} & \min(A_j) & A_j\setminus\{\min(A_j)\} & A_{j+1} & \cdots &  A_{k}\\
 a_1 & \cdots & a_m & v & a_{m+1} & \cdots & a_{j-2} & a_{j-1} & a_{j} & a_{j+1} & \cdots &  a_{k}\\
            \end{array}
\right) \\ \qquad \mbox{otherwise} 
\end{array} \right.
$$
and a partial transformation $\theta$ by 
$$
\theta=\left\{ \begin{array}{l}
\left(
          \begin{array}{ccccccccccc}
a_1 & \cdots & a_\ell & u & a_{\ell+1} & \cdots & a_m& v & a_{m+2} & \cdots &  a_{k}\\
a_1 & \cdots & a_\ell & u & a_{\ell+1} & \cdots & a_m& a_{m+1} & a_{m+2} & \cdots &  a_{k}\\
            \end{array}
\right) \\ \qquad \mbox{if $j=m+1$}
\\ \\
\left(
          \begin{array}{ccccccccccccccc}
 a_1 & \cdots & a_\ell & u & a_{\ell+1} & \cdots & a_m & v & a_{m+1} & \cdots & a_{j-1}  & a_{j+1} & \cdots &  a_{k}\\
 a_1 & \cdots & a_\ell & u & a_{\ell+1} & \cdots & a_m &  a_{m+1} & a_{m+2} & \cdots & a_{j}  & a_{j+1} & \cdots &  a_{k}\\
            \end{array}
\right) \\ \qquad \mbox{otherwise.} 
\end{array} \right.
$$
       
In each case, it is a routine matter to verify that $\beta,\widehat{\theta}\in\O_n(Y)$,    $|\im(\beta)| = |\im(\widehat{\theta})|= k+1$ and 
$\alpha=\beta\widehat{\theta}$, as required.        
\end{proof}

Now, let 
$$
A = \{\alpha \in \O_n(Y) \mid |\im(\alpha)| = r\}=\{\alpha \in \O_n(Y) \mid \im(\alpha)=Y\}
$$ 
and 
$$
B = \{\alpha \in \O_n(Y) \mid \mbox{$\alpha$ is regular and $|\im(\alpha)| = r-1$}\}.
$$ 

By Lemma \ref{max2} and a simple induction process, we may conclude that each element of $\O_n(Y)$ with rank less than or equal to $r-1$ is a product of elements of rank $r-1$, which in turn, by Lemma \ref{max1}, are products of elements of $A\cup B$. 
On the other hand, take $\alpha\in A$ and $\beta,\gamma\in\O_n(Y)$ such that $\alpha=\beta\gamma$. Then $\ker(\beta)\subseteq\ker(\alpha)$ and, since the ranks of $\beta$ and $\gamma$ cannot be smaller than the rank of $\alpha$, 
we must also have $\beta,\gamma\in A$. It follows that $\alpha$ and $\beta$ have the same image and kernel, whence $\alpha=\beta$ 
(since $\O_n$ is $\mathscr{H}$-trivial).  Thus, we immediately have: 

\begin{proposition}\label{ger1}
The semigroup $\O_n(Y)$ is generated by $A\cup B$. Moreover, any generating set of $\O_n(Y)$ must contain $A$. 
\end{proposition}

\smallskip 

Observe that, since the elements of $A$ have all the same image and $\O_n$ is $\mathscr{H}$-trivial, then 
$A$ has as much elements as the number of distinct kernels, i.e.  $|A|=\binom{n-1}{r-1}$, 
the number of \textit{convex} equivalences of weight $r$ on $\{1,\ldots,n\}$ (see \cite{Gomes&Howie:1992}). 

\medskip 

For each $i \in \{1, \ldots, r\}$, define  
$$
B_{i} = \{\alpha \in B \mid \im(\alpha) = Y\setminus\{y_{i}\}\}=\{\alpha \in B \mid Y\alpha = Y\setminus\{y_{i}\}\}.
$$
and 
$$
\varepsilon_{i} =
            \left(
              \begin{array}{cccccc}
                y_{1} & \cdots & y_{i-1} & y_{i+1} & \cdots & y_{r}\\
                y_{1} & \cdots & y_{i-1} & y_{i+1} & \cdots & y_{r}\\
              \end{array}
            \right).
$$
Clearly, $\widehat{\varepsilon}_{i}, \widetilde{\varepsilon}_{i} \in B_{i}$, for $1\le i\le r$, and 
$B = B_{1}\,\dot{\cup}\,B_{2}\,\dot{\cup}\,\cdots\,\dot{\cup}\,B_{r}$.  
Moreover, we have the following useful decompositions: 

\begin{lemma}\label{dec}
        Let $i\in \{1,\ldots,r\}$ and let $\alpha \in B_{i}$. The following statements hold: 
        \begin{enumerate}
        \item If $i \le r-1$ then there exists 
        $\beta \in B_{i+1}$ such that $\alpha = \beta\widetilde{\varepsilon}_{i}$; 
        \item If $i \ge 2$ then there exists $\beta \in B_{i-1}$ such that $\alpha = \beta\widehat{\varepsilon}_{i}$.
        \end{enumerate}
\end{lemma}
\begin{proof}
     Since $\im(\alpha) = Y \setminus \{y_{i}\} = Y\alpha$, then we can write
$\alpha =\left(
              \begin{array}{cccccc}
                A_{1} & \cdots & A_{i-1} & A_{i}   & \cdots & A_{r-1}\\
               y_{1} & \cdots & y_{i-1} & y_{i+1}  & \cdots & y_{r} \\
              \end{array}
            \right),$
      where $A_{i} \cap Y \neq \emptyset$, for all $i \in\{ 1, \ldots, r-1\}$.
If $i\le r-1$ then, being 
      $$\beta =
            \left(
              \begin{array}{ccccccc}
                A_{1} & \cdots & A_{i-1} & A_{i} & A_{i+1} & \cdots & A_{r-1}\\
               y_{1} & \cdots & y_{i-1} & y_{i} & y_{i+2} & \cdots & y_{r} \\
              \end{array}
            \right)\in B_{i+1},$$ 
      we have $\alpha = \beta\widetilde{\varepsilon}_{i}$. 
On the other hand, if $i\ge2$ then, being 
      $$\beta =
            \left(
              \begin{array}{ccccccc}
                A_{1} & \cdots & A_{i-2} & A_{i-1} & A_{i} & \cdots & A_{r-1}\\
               y_{1} & \cdots & y_{i-2} & y_{i} & y_{i+1} & \cdots & y_{r} \\
              \end{array}
            \right)\in B_{i-1},$$ 
      we have $\alpha = \beta\widehat{\varepsilon}_{i}$, as required.
\end{proof}

Now, let $i\in\{1,\ldots,r\}$ and take $\alpha\in B_k$, for some $k\in\{1,\ldots,r\}\setminus\{i\}$. Then, by the previous lemma, we have 
$$
\alpha=\left\{
\begin{array}{ll}
\beta \widetilde{\varepsilon}_{i-1}\cdots \widetilde{\varepsilon}_{k} & 
\mbox{if $k<i$}\\
\beta \widehat{\varepsilon}_{i+1}\cdots \widehat{\varepsilon}_{k} & 
\mbox{if $k>i$},\\ 
\end{array}
\right.
$$
for some $\beta\in B_i$. Hence, as a consequence of Proposition \ref{ger1} and Lemma \ref{dec}, we immediately obtain: 

\begin{corollary}\label{ger2}
For any $i\in\{1,\ldots,r\}$, one has $\O_n(Y)=\langle A, \widetilde{\varepsilon}_{1}, 
\ldots, \widetilde{\varepsilon}_{i-1},B_i, 
\widehat{\varepsilon}_{i+1},\ldots, \widehat{\varepsilon}_{r}\rangle$. 
\end{corollary}

\medskip 

Next, let $i=\min\{k\in\{1,\dots,n\}\mid k\not\in Y\}$. 
Clearly, we must have $1\le i\le r+1$. Moreover: 
\begin{enumerate}
\item If $i=1$ then $y_1>1$; 
\item If $i=2$ then $y_1=1\in Y^\sharp$ and $2<y_2\not\in Y^\sharp$; 
\item If $3\le i \le r$ then we have $y_1=1,\ldots,y_{i-1}=i-1$ and $y_i>i$, with $y_1,\ldots,y_{i-2}\in Y^\sharp$ and $y_{i-1},y_i\not\in Y^\sharp$; and  
\item If $i=r+1$ then $Y=\{1,\ldots,r\}$ and so $Y^\sharp=\{1,\ldots,r-1\}$. 
\end{enumerate} 

For the case $i=1$, we have: 

\begin{lemma}\label{imin}
If $y_1>1$ then $B_1\subseteq \langle A\rangle$. Moreover, in this case, $\O_n(Y)=\langle A, 
\widehat{\varepsilon}_{2},\ldots, \widehat{\varepsilon}_{r}\rangle$. 
\end{lemma}
\begin{proof} 
Take $\alpha =\left(
              \begin{array}{ccc}
                A_{2}  & \cdots & A_{r}\\
               y_{2}  & \cdots & y_{r} \\
              \end{array}
            \right)\in B_1$.
Since $y_1>1$, we must have $1,y_1\in A_2$. 
Define $$\beta =\left(
              \begin{array}{ccccc}
               1 & A_2\setminus\{1\} & A_{3}  & \cdots & A_{r}\\
              y_1 & y_{2}  & y_3 & \cdots & y_{r} \\
              \end{array}
            \right)
\quad\text{and}\quad  
\theta =\left(
              \begin{array}{ccccc}
               1 & y_1 & y_{3}  & \cdots & y_{r}\\
              y_1 & y_{2}  & y_3 & \cdots & y_{r} \\
              \end{array}
            \right).  
$$
Then, clearly $\beta,\widehat{\theta}\in A$ and $\alpha=\beta\widehat{\theta}$.
Hence $B_1\subseteq \langle A\rangle$ and thus, by Corollary \ref{ger2},  $\O_n(Y)=\langle A, \widehat{\varepsilon}_{2}, \ldots, \widehat{\varepsilon}_{r}\rangle$, as required. 
\end{proof}

Observe that $y_1\in Y^\sharp$ if and only if $y_1=1$ (if and only if $1\in Y$). 

Similarly to the previous lemma, we may prove:  

\begin{lemma}\label{imax}
If $y_r<n$ then $B_r\subseteq \langle A\rangle$. Moreover, in this case, $\O_n(Y)=\langle A, \widetilde{\varepsilon}_{1}, 
\ldots, \widetilde{\varepsilon}_{r-1}\rangle$. 
\end{lemma}

In particular, this lemma guarantees us that  
$\O_n(\{1,\ldots,r\})=\langle A, \widetilde{\varepsilon}_{1}, 
\ldots, \widetilde{\varepsilon}_{r-1}\rangle$. 

\begin{lemma}\label{imed}
If $y_1=1,\ldots,y_{i-1}=i-1$ and $y_i>i$, for some $2\le i \le r$, 
then $\O_n(Y)=\langle A, \widetilde{\varepsilon}_{1}, 
\ldots, \widetilde{\varepsilon}_{i-1},\widehat{\varepsilon}_{i},
\widehat{\varepsilon}_{i+1},\ldots, \widehat{\varepsilon}_{r}\rangle$.
\end{lemma}
\begin{proof} 
Take $\alpha =\left(
              \begin{array}{cccccc}
                A_{1}  & \cdots & A_{i-1} & A_{i+1} & \cdots & A_{r}\\
               y_{1}  & \cdots & y_{i-1} &  y_{i+1}  & \cdots & y_{r}\\
              \end{array}
            \right)\in B_i$. 
Then $A_j\cap Y\ne\emptyset$, for all $j=1,\ldots,i-1,i+1,\ldots,r$, whence 
$y_i\in A_{i-1}\cup A_{i+1}$ and so $i\in A_{i-2}\cup  A_{i-1}\cup A_{i+1}$ 
(where $A_{i+1}=\emptyset$, if $i=r$, and $A_{i-2}=\emptyset$, if $i=2$).

First, suppose that $i\in A_{i+1}$ (in this case, we must have $i<r$). 
Then, since $A_{i+1}\cap Y\ne\emptyset$, we must have $|A_{i+1}|\ge2$.  
Define 
$$
\beta =\left(
              \begin{array}{cccccccc}
  A_1 & \cdots & A_{i-1}  & \min(A_{i+1}) & A_{i+1}\setminus\{\min(A_{i+1})\} & A_{i+2} & \cdots & A_r\\
  y_1 & \cdots & y_{i-1} & y_i & y_{i+1} & y_{i+2} & \cdots & y_r\\
              \end{array}
            \right)
$$
and
$$    
\theta =\left(
              \begin{array}{cccccccc}
              y_1 & \cdots & y_{i-1} & i   & y_i & y_{i+2} & \cdots &  y_r\\
              y_1 & \cdots & y_{i-1} & y_i & y_{i+1} &  y_{i+2} & \cdots &  y_r \\
              \end{array}
            \right).  
$$
Hence, clearly $\beta,\widehat{\theta}\in A$ and $\alpha=\beta\widehat{\theta}$.

Next, suppose that $i\in A_{i-1}$. Then, since $A_{i-1}\cap Y\ne\emptyset$, 
we must also have $|A_{i-1}|\ge2$. Hence, by defining  
$$
\beta =\left(
              \begin{array}{cccccccc}
  A_1 & \cdots & A_{i-2}  & \min(A_{i-1}) & A_{i-1}\setminus\{\min(A_{i-1})\} & A_{i+1} & \cdots & A_r\\
  y_1 & \cdots & y_{i-2} & y_{i-1} & y_{i} & y_{i+1} & \cdots & y_r\\
              \end{array}
            \right)\in A,  
$$
we obtain $\alpha=\beta\widehat{\varepsilon}_{i}$. 

Finally, suppose that $i\in A_{i-2}$ (in this case, we must have $i>2$). 
Once again, since $A_{i-2}\cap Y\ne\emptyset$, it follows that $|A_{i-1}|\ge2$. 
Then, define 
$$
\beta =\left(
              \begin{array}{ccccccccc}
  A_1 & \cdots & A_{i-3}  & \min(A_{i-2}) & A_{i-2}\setminus\{\min(A_{i-2})\} 
  & A_{i-1} & A_{i+1} & \cdots & A_r \\
  y_1 & \cdots & y_{i-3} & y_{i-2} & y_{i-1} & y_i & y_{i+1} & \cdots & y_r\\
              \end{array}
            \right)
$$
and
$$    
\theta =\left(
              \begin{array}{cccccccc}
              y_1 & \cdots & y_{i-2} & i       & y_i & y_{i+1} & \cdots &  y_r\\
              y_1 & \cdots & y_{i-2} & y_{i-1} & y_i & y_{i+1} & \cdots &  y_r \\
              \end{array}
            \right).  
$$
Clearly, we have $\beta,\widehat{\theta}\in A$ and 
$\alpha=\beta\widehat{\theta}\widehat{\varepsilon}_{i}$.            

Therefore, we proved that $B_i\subseteq\langle A, \widehat{\varepsilon}_{i}\rangle$ and so, by Corollary \ref{ger2}, the result follows.            
\end{proof}

Now, for $3\le i \le r$, define
$$    
\varepsilon_{1,i} = \left(
              \begin{array}{cccccc}
              y_1 & \cdots & y_{i-2} & y_{i} & \cdots &  y_r\\
              y_2 & \cdots & y_{i-1} & y_{i} & \cdots &  y_r\\
              \end{array}
            \right). 
$$
Clearly, $\widetilde{\varepsilon}_{1,i} \in B_1$. 

\begin{lemma}\label{lemcr}
If $y_1=1,\ldots,y_{i-1}=i-1$ and $y_i>i$, for some $3\le i \le r$, 
then $\widetilde{\varepsilon}_{1}, 
\widetilde{\varepsilon}_{i-1}\in
\langle A,\widetilde{\varepsilon}_{1,i}\rangle$. 
\end{lemma}
\begin{proof}
Let $   
\theta =\left(
              \begin{array}{ccccccc}
    y_2 & \cdots & y_{i-1} & y_{i-1}+1 & y_i &\cdots &  y_r\\
    y_1 & \cdots & y_{i-2} & y_{i-1} & y_i & \cdots &  y_r \\
              \end{array}
            \right).  
$
Then $\widetilde{\theta}\in A$ and it is easy to verify that 
$\widetilde{\varepsilon}_{1}=\widetilde{\theta}\widetilde{\varepsilon}_{1,i}$ 
and 
$\widetilde{\varepsilon}_{i-1}=\widetilde{\varepsilon}_{1,i}\widetilde{\theta}$, 
which proves the lemma. 
\end{proof}

Next, for $2\le j \le r-1$, define
$$    
\varepsilon_{r,j} = \left(
              \begin{array}{cccccc}
              y_1 & \cdots & y_{j-1} & y_{j+1} & \cdots &  y_r\\
              y_1 & \cdots & y_{j-1} & y_{j} & \cdots &  y_{r-1}\\
              \end{array}
            \right). 
$$
It is clear that $\widehat{\varepsilon}_{r,j} \in B_r$. Moreover, similarly to the previous lemma, we may prove:  

\begin{lemma}\label{lemc}
If $y_r=n,\ldots,y_j=n-r+j$ and $y_{j-1}<n-r+j-1$, for some $2\le j \le r-1$, 
then $\widehat{\varepsilon}_{r}, 
\widehat{\varepsilon}_{j}\in
\langle A,\widehat{\varepsilon}_{r,j}\rangle$. 
\end{lemma}

Before proving Theorem \ref{main}, we still need two more lemmas. 

\begin{lemma}\label{lema}
If $y_k+1<y_{k+1}$, for some $2\le k \le r-1$, 
then $\widehat{\varepsilon}_{k}\in\langle A\rangle$. 
\end{lemma}
\begin{proof}
Let $   
\theta =\left(
              \begin{array}{ccccccc}
    y_1 & \cdots & y_{k-1} & y_{k}+1 & y_{k+1} &\cdots &  y_{r}\\
    y_1 & \cdots & y_{k-1} & y_{k} & y_{k+1} & \cdots &  y_r \\
              \end{array}
            \right).  
$
Then $\widehat{\theta}\in A$ and we have 
$\widehat{\varepsilon}_{k}=\widehat{\theta}\,^2\in\langle A\rangle$,  
as required. 
\end{proof}

\begin{lemma}\label{lemb}
If $y_{k-1}<y_k-1$ and $y_k<n-r+k$, for some $2\le k \le r-1$, 
then $\widehat{\varepsilon}_{k}\in\langle A\rangle$. 
\end{lemma}
\begin{proof}
Since $y_k<n-r+k$ then $n-y_k+1>r-k+1$, 
whence $\{y_k,\ldots,y_r\}\subsetneq\{y_k,\ldots,n\}$. 
Take $y\in \{y_k,\ldots,n\}\setminus\{y_k,\ldots,y_r\}$ 
and let $\ell\in\{k,\ldots,r\}$ be such that $y_\ell<y<y_{\ell+1}$ 
(where $y_{\ell+1}=n+1$, if $\ell=r$). 
If $\ell=k$ then $y_k+1\le y<y_{k+1}$ and so, by Lemma \ref{lema}, 
$\widehat{\varepsilon}_{k}\in\langle A\rangle$. 
On the other hand, suppose that $\ell\ge k+1$ and define 
$$   
\theta_1 =\left(
              \begin{array}{cccccccccc}
 y_1 & \cdots & y_{k-1} & y_{k+1} & \cdots & y_\ell & y & y_{\ell+1} & \cdots &  y_{r}\\
 y_1 & \cdots & y_{k-1} & y_{k} & \cdots & y_{\ell-1} & y_\ell & y_{\ell+1} & \cdots &  y_r \\
              \end{array}
            \right)
$$
and
$$ 
\theta_2 =\left(
              \begin{array}{ccccccccccc}
 y_1 & \cdots & y_{k-1} & y_k-1 & y_k    & y_{k+1} & \cdots & y_{\ell-1} & y_{\ell+1} & \cdots &  y_{r}\\
 y_1 & \cdots & y_{k-1} & y_k & y_{k+1}  & y_{k+2} & \cdots & y_{\ell}  & y_{\ell+1} & \cdots &  y_r \\
              \end{array}
            \right).             
$$
Then $\widehat{\theta}_1,\widehat{\theta}_2\in A$ and it is easy to show that 
$\widehat{\varepsilon}_{k}=\widehat{\theta}_1\widehat{\theta}_2\in\langle A\rangle$,  
as required. 
\end{proof}

Finally, we are now able to prove the main result of this section. 
    
\begin{proof}[Proof of Theorem \ref{main}]
Let $i=\min\{k\in\{1,\ldots,n\}\mid k\not\in Y\}$. Recall that $1\le i\le r+1$. 

If $i=r+1$ then $Y=\{1,\ldots,r\}$, whence $Y^\sharp=\{1,\ldots,r-1\}$. Moreover, as we already observed, Lemma \ref{imax} ensures us that 
$\O_n(Y)=\langle A, \widetilde{\varepsilon}_{1}, 
\ldots, \widetilde{\varepsilon}_{r-1}\rangle$ and so, in this case, we have a generating set of $\O_n(Y)$ with $|A|+|Y^\sharp|$ elements. 

Next, suppose that $1\le i\le r$. Then, by Lemmas \ref{imin}, \ref{imed} and \ref{lemcr}, we have 
\begin{equation}\label{genI}
\O_n(Y)=\left\{
\begin{array}{ll}
\langle A, \widehat{\varepsilon}_{2},\ldots, \widehat{\varepsilon}_{r}\rangle & 
\mbox{if $i=1$}\\
\langle A, \widetilde{\varepsilon}_{1},\widehat{\varepsilon}_{2},\ldots, \widehat{\varepsilon}_{r}\rangle & 
\mbox{if $i=2$} \\
\langle A, \widetilde{\varepsilon}_{1,i}, \widetilde{\varepsilon}_{2}, 
\ldots, \widetilde{\varepsilon}_{i-2},\widehat{\varepsilon}_{i},
\widehat{\varepsilon}_{i+1},\ldots, \widehat{\varepsilon}_{r}\rangle & 
\mbox{if $3\le i\le r$}. 
\end{array}
\right.
\end{equation}
Notice that, we have $y_1=1\in Y^\sharp$, for $i=2$, 
and $\{y_1,\ldots,y_{i-2}\}=\{1,\ldots,i-2\}\subseteq Y^\sharp$, for $3\le i\le r$. 

Now, let $j\in\{1,\ldots,r+1\}$ be such that  
$n-r+j-1\not\in Y$ and $\{n-r+j,\ldots,n\}\subseteq Y$. 
Notice that $i<n-r+j$, since $\{1,\ldots,i-1\}\cup \{n-r+j,\ldots,n\}\subseteq Y\subsetneq \{1,\ldots,n\}$. Particularly, it follows that 
$\{1,\ldots,i-1\}\cap\{n-r+j,\ldots,n\}=\emptyset$, whence  $(i-1)+(r-j+1)\le |Y|=r$
and so $i\le j$. 

If $j=1$ then $Y=\{n-r+1,\ldots,n\}$ and thus $Y^\sharp=\{n-r+2,\ldots,n\}$. 
Furthermore, in this case, we also have $i=1$ and so $A\cup\{\widehat{\varepsilon}_{2},\ldots, \widehat{\varepsilon}_{r}\}$ is a generating set of $\O_n(Y)$ with $|A|+|Y^\sharp|$ elements. 

If $j=r+1$ then $n\not\in Y$, whence $y_r<n$ (and so $y_r\not\in Y^\sharp$). 
Thus, by Lemma \ref{imax}, from (\ref{genI}) we get 
\begin{equation}\label{genII}
\O_n(Y)=\left\{
\begin{array}{ll}
\langle A, \widehat{\varepsilon}_{2},\ldots, \widehat{\varepsilon}_{r-1}\rangle & 
\mbox{if $i=1$}\\
\langle A, \widetilde{\varepsilon}_{1},\widehat{\varepsilon}_{2},\ldots, \widehat{\varepsilon}_{r-1}\rangle & 
\mbox{if $i=2$} \\
\langle A, \widetilde{\varepsilon}_{1,i}, \widetilde{\varepsilon}_{2}, 
\ldots, \widetilde{\varepsilon}_{i-2},\widehat{\varepsilon}_{i},
\widehat{\varepsilon}_{i+1},\ldots, \widehat{\varepsilon}_{r-1}\rangle & 
\mbox{if $3\le i\le r$}. 
\end{array}
\right.
\end{equation}

If $j=r$ then $n\in Y$ and so $y_r=n\in Y^\sharp$. 

Finally, if $2\le j\le r-1$ then, by Lemma \ref{lemc}, 
from (\ref{genI}) we obtain  
\begin{equation}\label{genIII}
\O_n(Y)=\left\{
\begin{array}{ll}
\langle A, \widehat{\varepsilon}_{2},\ldots, 
\widehat{\varepsilon}_{j-1},\widehat{\varepsilon}_{j+1},\ldots,\widehat{\varepsilon}_{r-1}, 
\widehat{\varepsilon}_{r,j}\rangle & 
\mbox{if $i=1$}\\
\langle A, \widetilde{\varepsilon}_{1},\widehat{\varepsilon}_{2},\ldots,
\widehat{\varepsilon}_{j-1},\widehat{\varepsilon}_{j+1},\ldots,\widehat{\varepsilon}_{r-1}, 
\widehat{\varepsilon}_{r,j}\rangle & 
\mbox{if $i=2$} \\
\langle A, \widetilde{\varepsilon}_{1,i}, \widetilde{\varepsilon}_{2}, 
\ldots, \widetilde{\varepsilon}_{i-2},\widehat{\varepsilon}_{i},
\ldots, \widehat{\varepsilon}_{j-1},\widehat{\varepsilon}_{j+1},\ldots,\widehat{\varepsilon}_{r-1}, 
\widehat{\varepsilon}_{r,j}\rangle & 
\mbox{if $3\le i\le r$}. 
\end{array}
\right.
\end{equation}
Notice that, in this case, we have 
$\{y_{j+1},\ldots,y_r\}=\{n-r+j+1,\ldots,n\}\subseteq Y^\sharp$. 

Now, observe that if $y_k=n-r+k$, for some $1\le k\le r$, 
then $y_\ell=n-r+\ell$, for all $k\le \ell\le r$. 

Thus, let us take an element of the type $\widehat{\varepsilon}_{k}$ from the sets of generators (\ref{genII}), for $j=r+1$, or from the sets of generators 
(\ref{genI}), for $j=r$, or from the sets of generators (\ref{genIII}),  
for $2\le j\le r-1$, such that $y_k\not\in Y^\sharp$. 
Hence $2\le k\le r-1$ and, from what we observed above and the definition of $j$, 
we must have $y_k<n-r+k$. 
Moreover, since $y_k\not\in Y^\sharp$, $y_k+1<y_{k+1}$ or $y_{k-1}<y_k-1$ and so, by Lemma \ref{lema} or by Lemma \ref{lemb}, respectively, 
we have $\widehat{\varepsilon}_{k}\in\langle A\rangle$. Therefore, in each case, by removing all the elements of this type from the considered generating set, we obtain a new generating set of $\O_n(Y)$ with exactly $|A|+|Y^\sharp|$ elements. 

So far we proved that $\rank(\O_n(Y))\le |A|+|Y^\sharp|$. Next, we will prove the opposite inequality. 

Let $D$ be any generating set of $\O_n(Y)$. Then, by Proposition \ref{ger1}, we have  $A\subseteq D$. Now, for each $j \in \{1, \ldots, r\}$, define  
$
C_{j} = \{\alpha \in \O_n(Y) \mid \im(\alpha) = Y\setminus\{y_{j}\}\}
$
and let $C=C_1\,\dot{\cup}\,\cdots\,\dot{\cup}\,C_r$. 
Hence, by showing that $C_j\cap D\ne\emptyset$, for all $j \in \{1, \ldots, r\}$ such that $y_j\in Y^\sharp$, we get $\rank(\O_n(Y))\ge |A|+|Y^\sharp|$ and so  $\rank(\O_n(Y)) = |A|+|Y^\sharp|$, as required. 
Therefore, let $j \in \{1, \ldots, r\}$ be such that $y_j\in Y^\sharp$. 

First, suppose that $j\in\{1,r\}$. Then $y_j\in\{1,n\}$ and so any transformation $\beta\in\O_n$ such that $y_j\in\im(\beta)$ must fix $y_j$. 
Let $\alpha$ be any element of $C_j$ and let $\beta_1,\ldots,\beta_k\in D$ ($k\in\N$) be such that $\alpha=\beta_1\cdots\beta_k$. 
Since $\alpha$ has rank $r-1$, we deduce that $\beta_1,\ldots,\beta_k\in A\cup C$. If $\beta_1,\ldots,\beta_k\in (A\cup C)\setminus C_j$ then $y_j\in\im(\beta_i)$, for all $i \in \{1, \ldots, k\}$, and so 
$\beta_i$ fixes $y_j$, for all $i \in \{1, \ldots, k\}$, whence $\alpha=\beta_1\cdots\beta_k$ fixes $y_j$, which is a contradiction. 
Thus $\beta_i\in C_j$, for some $i \in \{1, \ldots, k\}$, 
and so $C_j\cap D\ne\emptyset$.  

Now, assume that $1<j<r$. Then $y_{j-1}=y_j-1$ and $y_{j+1}=y_j+1$.
Take 
$$ 
\alpha =\left(
              \begin{array}{cccccccccc}
 1 & \cdots & j-1 & j & j+1  & \cdots & r-1 & r & \cdots &  n\\
 y_1 & \cdots & y_{j-1} & y_{j+1}  & y_{j+2} & \cdots & y_r  & y_r & \cdots &  y_r \\
              \end{array}
            \right).             
$$
Then $\alpha\in C_j$ and, as above, there exist $k\in\N$ and $\beta_1,\ldots,\beta_k\in D\cap (A\cup C)$ such that $\alpha=\beta_1\cdots\beta_k$. 
If $k=1$ then $\beta_1=\alpha\in C_j\cap D$. So, suppose that $k>1$. 
Moreover, we may assume that $\alpha\ne\beta_1\cdots\beta_i$, for $i=1,\ldots,k-1$. 

Let $\beta=\beta_1\cdots\beta_{k-1}$. 
Then $\beta\in A\cup C$ and, as $\alpha$ is injective in $\{1,\ldots,r-1\}$, so is $\beta$. 

Suppose that $\beta\in A$. Then 
$$ 
\beta =\left(
              \begin{array}{ccccccccccccc}
 1 & \cdots & j-1 & j & j+1  & \cdots & r-1 & r & \cdots & r-1+t & r+t & \cdots & n\\
 y_1 & \cdots &  y_{j-1} &y_j & y_{j+1}   & \cdots & y_{r-1}  & y_{r-1} & \cdots & y_{r-1} & y_r & \cdots & y_r \\
              \end{array}
            \right),              
$$
for some $0\le t\le n-r$, and so, as $\alpha=\beta\beta_k$, the restriction of $\beta_k$ to $Y$ is transformation 
$$ 
\theta=\left(
              \begin{array}{cccccccc}
 y_1 & \cdots & y_{j-1} & y_j & y_{j+1}  & \cdots & y_{r-1} & y_r \\
 y_1 & \cdots & y_{j-1} & y_{j+1}  & y_{j+2} & \cdots & y_r  & y_r  \\
              \end{array}
            \right).             
$$
If $\beta_k\in A$ then there exists $y_{j-1}<y<y_j$ such that $y\beta_k=y_j$, which contradicts the equality $y_{j-1}=y_j-1$. Thus $\beta_k\in C_j$ and so $C_j\cap D\ne\emptyset$. 

Finally, suppose that $\beta\in C$. Then 
$$ 
\beta =\left(
              \begin{array}{cccccccccc}
 1 & \cdots & j-1 & j & j+1  & \cdots & r-1 & r & \cdots & n\\
 y_1 & \cdots &  y_{j-1} & y_j & y_{j+1} & \cdots & y_{r-1}  & y_{r-1}& \cdots & y_{r-1} \\
              \end{array}
            \right)        
$$
or  
$$ 
\beta =\left(
              \begin{array}{cccccccccc}
 1 & \cdots & \ell-1 & \ell & \ell+1  & \cdots & r-1 & r & \cdots & n\\
 y_1 & \cdots &  y_{\ell-1} & y_{\ell+1} & y_{\ell+2} & \cdots & y_{r}  & y_{r}& \cdots & y_{r} \\
              \end{array}
            \right),              
$$
for some $1\le\ell\le r-1$. In the first case, the restriction of $\beta_k$ to $Y$ must be again the transformation $\theta$ defined above, which once again implies that $\beta_k\in C_j$ and so $C_j\cap D\ne\emptyset$. So, suppose we have the second case. First, observe that, as $\beta\ne\alpha$, then $\ell\ne j$. 

If $j<\ell\le r-1$ then the restriction of $\beta_k$ to $Y\setminus\{y_\ell\}$ is transformation 
$$ 
\left(
              \begin{array}{cccccccccc}
 y_1 & \cdots & y_{j-1} & y_j &      y_{j+1}  & \cdots & 
 y_{\ell-1} & y_{\ell+1} & \cdots & y_r \\
 y_1 & \cdots & y_{j-1} & y_{j+1}  & y_{j+2} & \cdots  & 
 y_\ell & y_{\ell+1} &  \cdots  & y_r  \\
              \end{array}
            \right)         
$$
and so, as above, $\beta_k\in A$ implies that there is $y_{j-1}<y<y_j$ such that $y\beta_k=y_j$, which contradicts, once more, the equality $y_{j-1}=y_j-1$. Thus, also in this situation, $\beta_k\in C_j$ and so $C_j\cap D\ne\emptyset$. 

On the other hand, if $1\le\ell<j$ then the restriction of $\beta_k$ to $Y\setminus\{y_\ell\}$ is transformation 
$$ 
\left(
              \begin{array}{cccccccccc}
 y_1 & \cdots & y_{\ell-1} & y_{\ell+1}  & y_{\ell+2} & \cdots & y_j    & 
 y_{j+1} & \cdots & y_r \\
 y_1 & \cdots & y_{\ell-1} & y_{\ell}  & y_{\ell+1} & \cdots  & y_{j-1} & 
 y_{j+1} &  \cdots  & y_r  \\
              \end{array}
            \right),           
$$
which also implies that $\beta_k\in C_j$. In fact, if $\beta_k\in A$ then there exists $y_j<y<y_{j+1}$ such that $y\beta_k=y_j$, which contradicts, this time, the equality $y_{j+1}=y_j+1$. Thus, under these conditions, also $C_j\cap D\ne\emptyset$. 

Therefore, we proved that $C_j\cap D\ne\emptyset$, for all $j \in \{1, \ldots, r\}$ such that $y_j\in Y^\sharp$, as required. 
\end{proof}

Observe that $\O_n(Y)$ is generated by $A$ if and only if $Y^\sharp=\emptyset$. Moreover, in this case, $$\rank(\O_n(Y))= |A|=\binom{n-1}{r-1}.$$

\paragraph*{Acknowledgments.} This research was mainly carried out during the visit of the second and fourth authors to Faculdade de Ci\^encias e Tecnologia da Universidade Nova de Lisboa and Centro de  \'Algebra da Universidade de Lisboa  between August and October 2011.

The authors would like to thank Francisco and Cl\'audia Coelho for their help in reviewing the text of this paper. 


\lastpage 


\begin{thebibliography}{00}

\bibitem{Adams&Gould:1989} 
M.E. Adams and M. Gould, 
Posets whose monoids of order-preserving maps are regular, 
Order 6 (1989), 195--201.

\bibitem{Aizenstat:1962} 
A.Ya. A\u{\i}zen\v{s}tat,
The defining relations of the endomorphism semigroup of 
a finite linearly ordered set,
Sibirsk. Mat. 3 (1962), 161--169 (Russian). 

\bibitem{Aizenstat:1962b} 
A.Ya. A\u{\i}zen\v{s}tat, 
Homomorphisms of semigroups of endomorphisms of ordered sets, 
Uch. Zap., Leningr. Gos. Pedagog. Inst.  238 (1962), 38--48 (Russian). 

\bibitem{Aizenstat:1968} 
A.Ya. A\u{\i}zen\v{s}tat,  
Regular semigroups of endomorphisms of ordered sets, 
Leningrad. Gos. Ped. Inst. U\v cen. Zap. 387 (1968), 3--11 (Russian). 

\bibitem{Almeida&Volkov:1998} 
J. Almeida and M.V. Volkov, 
The gap between partial and full, 
Int. J. Algebra Comput. 8 (1998), 399--430.

\bibitem{Fernandes:1997} 
V.H. Fernandes, 
Semigroups of order-preserving mappings on a finite chain: a new class of divisors,
Semigroup Forum  54 (1997), 230--236.

\bibitem{Fernandes:2002} 
V.H. Fernandes, 
Semigroups of order-preserving mappings on a finite chain: another class of divisors, 
Izvestiya VUZ. Matematika 3 (478) (2002), 51--59 (Russian).

\bibitem{Fernandes&al:2010} 
V.H. Fernandes, M.M. Jesus, V. Maltcev and J.D. Mitchell, 
Endomorphisms of the semigroup of order-preserving mappings, 
Semigroup Forum 81 (2010), 277--285. 

\bibitem{Fernandes&Sanwong:2012} 
V.H. Fernandes and J. Sanwong, 
On the rank of semigroups of transformations on a finite set with restricted range, 
Algebra Colloq. (to appear). 

\bibitem{Fernandes&Volkov:2010} 
V.H. Fernandes and M.V. Volkov, 
On divisors of semigroups of order-preserving mappings of a finite chain, 
Semigroup Forum 81 (2010), 551--554. 

\bibitem{Gomes&Howie:1992}
G.M.S. Gomes and J.M. Howie, 
On the ranks of certain semigroups of order-preserving transformations, 
Semigroup Forum 45 (1992), no. 3, 272--282. 

\bibitem{Hall:1967} 
M. Hall Jr., 
Combinatorial Theory, 
John Wiley \& Sons, New York, 1967.

\bibitem{Higgins:1995} 
P.M. Higgins, 
Divisors of semigroups of order-preserving mappings on a finite chain, 
Int. J. Algebra Comput. 5 (1995), 725--742.

\bibitem{Howie:1971} 
J.M. Howie, 
Products of idempotents in certain semigroups of transformations,  
Proc. Edinburgh Math. Soc. (2) 17 (1971), 223--236. 

\bibitem{Howie:1995} 
J.M. Howie, 
Fundamentals of Semigroup Theory, 
Oxford, Oxford University Press, 1995.

\bibitem{Jitjankarn:2012} 
P. Jitjankarn, 
Isomorphism Theorems for Semigroups of Order-preserving Full Transformations, 
arXiv:1202.2977v1 [math.RA] 14 Feb 2012. 

\bibitem{Kemprasit&Changphas:2000}
Y. Kemprasit and T. Changphas, 
Regular order-preserving transformation semigroups, 
Bull. Austral. Math. Soc. 62 (2000), no. 3, 511--524. 

\bibitem{Kim&Kozhukhov:2008}
V.I. Kim and I.B. Kozhukhov, 
Regularity conditions for semigroups of isotone transformations of countable chains (Russian), 
Fundam. Prikl. Mat. 12 (2006), no. 8, 97--104; 
translation in J. Math. Sci. (N. Y.) 152 (2008), no. 2, 203--208. 

\bibitem{Goncalves&Sullivan:2011} 
S. Mendes-Gon\c calves and R.P. Sullivan, 
The ideal structure of semigroups of transformations with restricted range, 
Bull. Austral. Math. Soc. 83 (2011) 289--300.

\bibitem{Mora&Kemprasit:2010}
W. Mora and Y. Kemprasit,  
Regular elements of some order-preserving transformation semigroups,  
Int. J. Algebra 4 (2010), no. 13-16, 631--641. 

\bibitem{Nenthein&et.al:2005} 
S. Nenthein, P. Youngkhong, Y. Kemprasit, 
Regular elements of some transformation semigroups, 
Pure Math. Appl. 16 (2005), no. 3, 307--314.

\bibitem{Repnitskii&Vernitskii:2000} 
V.B. Repnitski\u{\i} and A. Vernitskii,  
Semigroups of order preserving mappings, 
Comm. in Algebra 28 (2000), No.8, 3635--3641.
       
\bibitem{Repnitskii&Volkov:1998} 
V.B. Repnitski\u{\i} and M.V. Volkov, 
The finite basis problem for the pseudovariety $\mathcal{O}$,  
Proc. R. Soc. Edinb., Sect. A, Math. 128 (1998), 661--669. 

\bibitem{Sanwong&Sommanee:2008} 
J. Sanwong and W. Sommanee, 
Regularity and Green's relations on a semigroup of transformations with restricted range,  
Int. J. Math. Math. Sci. 2008, Art. ID 794013, 11 pp.

\bibitem{Sanwong&et.al:2009} 
J. Sanwong, B. Singha and R.P. Sullivan, 
Maximal and minimal congruences on some semigroups, 
Acta Math. Sin. (Engl. Ser.) 25(3) (2009) 455--466. 

\bibitem{Sullivan:2008} 
R.P. Sullivan,
Semigroups of linear transformations with restricted range, 
Bull. Austral. Math. Soc. 77 (2008) 441--453.

\bibitem{Symons:1975} 
J.S.V. Symons, 
Some results concerning a transformation semigroup, 
J. Austral. Math. Soc. 19 (Series A) (1975) 413--425.

\bibitem{Vernitskii&Volkov:1995} 
A. Vernitskii and M.V. Volkov, 
A proof and generalisation of Higgins' division theorem for semigroups of order-preserving mappings, 
Izv.vuzov. Matematika, No.1 (1995), 38--44 (Russian).

\end{thebibliography}
\end{document}